\documentclass[reqno,12pt]{amsart}

\usepackage{amsmath,slashed,upgreek,bbm,appendix,tikz-cd,verbatim}
\usepackage[protrusion=true,babel=true]{microtype}
\usepackage[english]{babel}
\usepackage[widespace]{fourier}
\usepackage[backrefs]{amsrefs}
\usepackage[margin=1in]{geometry}
\usepackage[onehalfspacing]{setspace}
\usepackage[pdfusetitle,pagebackref]{hyperref}
\numberwithin{equation}{section}                
\usepackage[nameinlink,noabbrev]{cleveref}
\expandafter\def\csname ver@etex.sty\endcsname{3000/12/31}

\usepackage{autonum}                            
\usepackage{tikz}

\allowdisplaybreaks[1]

\let\originalleft\left
\let\originalright\right
\renewcommand{\left}{\mathopen{}\mathclose\bgroup\originalleft}
\renewcommand{\right}{\aftergroup\egroup\originalright}

\def\({\mathopen{}\left(}
\def\){\right)\mathclose{}}

\makeatletter
\renewcommand*{\eqref}[1]{\hyperref[{#1}]{\textup{\tagform@{\ref*{#1}}}}}
\makeatother

\newcommand*{\eqdef}{\mathrel{\vcenter{\baselineskip0.5ex \lineskiplimit0pt\hbox{.}\hbox{.}}}=}
\newcommand*{\defeq}{=\mathrel{\vcenter{\baselineskip0.5ex \lineskiplimit0pt\hbox{.}\hbox{.}}}}

\newtheorem*{acknowledgment}{Acknowledgment}
\newtheorem{theorem}{Theorem}[section]

\newtheorem{lemma}[theorem]{Lemma}

\newtheorem{remark}[theorem]{Remark}
\newtheorem{definition}[theorem]{Definition}

\crefname{theorem}{Theorem}{Theorems}                 
\creflabelformat{theorem}{#2{#1}#3}                   
\crefname{lemma}{Lemma}{Lemmas}                       
\creflabelformat{lemma}{#2{#1}#3}                     
\crefname{corollary}{Corollary}{Corollaries}          
\creflabelformat{corollary}{#2{#1}#3}                 
\crefname{ineq}{inequality}{inequalities}             
\creflabelformat{ineq}{#2{\upshape(#1)}#3}            
\crefname{diag}{diagram}{diagrams}                    
\creflabelformat{diag}{#2{\upshape(#1)}#3}            
\crefname{cond}{condition}{conditions}                
\creflabelformat{cond}{#2(#1)#3}                      
\crefname{table}{Table}{Tables}                       
\creflabelformat{table}{#2{\upshape(#1)}#3}           
\crefname{hypothesis}{Hypothesis}{Hypotheses}         
\creflabelformat{hypothesis}{#2{#1}#3}                
\crefname{remark}{Remark}{Remarks}                    
\creflabelformat{remark}{#2{#1}#3}                    
\crefname{definition}{Definition}{Definitions}        
\creflabelformat{def}{#2{#1}#3}                       
\crefname{conjecture}{Conjecture}{Conjectures}        
\creflabelformat{conj}{#2{#1}#3}                      

\def\id{\mathbbm{1}}
\def\cx{\mathbb{C}}
\def\bH{\mathbb{H}}

\def\rl{\mathbb{R}}
\def\bT{\mathbb{T}}
\def\Z{\mathbb{Z}}

\def\cB{\mathcal{B}}

\def\cD{\mathcal{D}}
\def\cE{\mathcal{E}}
\def\cF{\mathcal{F}}
\def\cG{\mathcal{G}}

\def\cI{\mathcal{I}}

\def\cL{\mathcal{L}}
\def\cM{\mathcal{M}}
\def\cN{\mathcal{N}}

\def\cR{\mathcal{R}}

\def\cU{\mathcal{U}}
\def\cV{\mathcal{V}}

\def\Ad{\mathrm{Ad}}

\def\coker{\mathrm{coker}}
\def\rd{\mathrm{d}}
\def\diag{\mathrm{diag}}

\def\dom{\mathrm{dom}}
\def\End{\mathrm{End}}
\def\rG{\mathrm{G}}

\def\Hom{\mathrm{Hom}}

\def\image{\mathrm{image}}
\def\index{\mathrm{index}}

\def\Re{\mathrm{Re}}

\def\S{\mathrm{S}}

\def\SO{\mathrm{SO}}
\def\Sp{\mathrm{Sp}}
\def\Spec{\mathrm{Spec}}
\def\spn{\mathrm{span}}
\def\Spin{\mathrm{Spin}}

\def\SU{\mathrm{SU}}

\def\tr{\mathrm{tr}}
\def\rU{\mathrm{U}}
\def\YMH{\mathrm{YMH}}
\def\vol{\: \mathrm{vol}}

\def\fB{\mathfrak{B}}
\def\fC{\mathfrak{C}}

\def\rd{\operatorname{d\!}{}}
\def\del{\partial}
\def\D{\slashed{D}}

\def\s{\mathfrak{s}}

\def\su{\mathfrak{su}}
\def\u{\mathfrak{u}}

\title{On the construction of monopoles with arbitrary symmetry breaking}
\date{\today}

\author{Benoit Charbonneau}
\address[Benoit Charbonneau]{Department of Pure Mathematics and Department of Physics and Astronomy, University of Waterloo, Ontario, Canada}
\urladdr{\href{https://www.math.uwaterloo.ca/~bcharbon/}{www.math.uwaterloo.ca/~bcharbon}}
\email{\href{mailto:benoit@alum.mit.edu}{benoit@alum.mit.edu}}

\author{\'Akos Nagy}
\address[\'Akos Nagy]{University of California, Santa Barbara}
\urladdr{\href{https://akosnagy.com/}{akosnagy.com}}
\email{\href{mailto:contact@akosnagy.com}{contact@akosnagy.com}}

\calclayout
\hypersetup{
   unicode        = true,
   pdffitwindow   = true,
   pdftoolbar     = false,
   pdfmenubar     = false,
   pdfstartview   = {FitH},
   hypertexnames  = true,
   colorlinks     = true,
   linkcolor      = black,
   citecolor      = black,
   filecolor      = black,
   urlcolor       = blue
}

\begin{document}

\begin{abstract}
   We produce finite energy BPS monopoles with arbitrary prescribed symmetry breaking from a new class of Nahm data.
\end{abstract}

\maketitle

\section{Introduction}

Gauge theory is for many of us the study of vector bundles equipped with exceptional objects, such as connections and other fields satisfying certain equations. Amongst the most classical of these objects are \emph{monopoles} on 3-manifolds, where the data given by a connection $\nabla$ with curvature $F_\nabla$ on a vector bundle $E$ and an endomorphism $\Phi$ of that bundle, called the \emph{Higgs field}, together being a solution to the \emph{monopole equation}:
\begin{equation}
   F_\nabla = \ast \nabla \Phi.
\end{equation}

The study of monopoles has been going on for quite some time, and yet, the most basic space on which to study monopoles, that is when the 3-manifold is the Euclidean space $\rl^3$, still offers many mysteries. Even when the structure group is the simplest, $\SU (2)$, there is still a lot to uncover. We certainly know a lot about the moduli space thanks to \cites{atiyah-hitchin,Corrigan1981,corrigan-goddard,DonaldsonSU2,hitchinGeodesics,hitchinMonopoles,hurtubiseNoteThmDonaldson,HurtubiseCharge2,JaffeTaubes,nakajima}, but Sen's conjectures \cite{Sen-Dyon-monopole-bound-states} that propose an understanding of the $L^2$-cohomology of the moduli space, have yet to be completely proven, although important steps in this direction have been made recently; see \cites{Fritzsch-Kottke-Singer-Monopoles-Sen-1,Kottke-Singer-partial-compactification-monopoles,Segal-Selby-cohomology-monopoles}. Understanding monopoles themselves, for instance the locations where Higgs fields are small, is the subject of the magnetic bag conjecture originally formulated by Bolognesi \cite{Bolognesi-MagneticBags} in 2006. While good progress has been made \cites{Bielawski-clusters,Evslin-Gudnason-magnetic-bags,Harland-largeN,LeeWeinberg-Monopole-Bags,Manton-Monopole-planets-galaxies,Taubes-Magnetic-bag}, it has not been resolved either.

When one cranks up the rank and is willing to consider the bigger structure group $\SU (N)$, there is a more vast field on which to play, even while we still stick to $\rl^3$ as our base manifold. This playing field was explored using multiple tools: the Nahm transform \cites{hurtubiseClassification,hurtubise_construction_1989,nahmEqn}, spectral curves \cites{hurtubise_monopoles_1990,Murray-nonabelian}, rational maps \cites{JarRatMapMpolesRadScatt,JarRatMapMpoles,JarConstrctnEuclidnMpoles,Murray-stratified,Schult-rationalmaps}. The moduli space's dimension was computed numerous times (see for instance \cites{Bielawski-asymptotic-metric-maxbreaking,Galan-monopoles-thesis,Mendizabal-HK,Moore-Royston-VandenBleeken-parameter-count-singular}) and its hyper-K\"ahler properties studied (see for instance \cites{Bielawski-asymptotic-metric-maxbreaking,Galan-monopoles-thesis,GRG97-Gibbons-Rychenkova-Goto,Mendizabal-HK}).

It is generally assumed, although not universally proven yet, that (under reasonable assumptions, such as finite energy), the Higgs field splits the bundle into eigenbundles at infinity. This splitting is called \emph{symmetry breaking}. To be more precise (and even more precision is shown in \Cref{sec:BPS}), we have in some gauge and in one direction an asymptotic behavior
\begin{equation}
   \Phi = \mu - \tfrac{1}{2 r} \kappa + O \( r^{- 2} \), \label{eq:phi_asymp}
\end{equation}
with $\mu, \kappa \in \su (N)$ satisfying $\left[ \mu, \kappa \right] = 0$. Thus, we can simultaneously diagonalize $\mu$ and $\kappa$. In the case of \emph{maximal symmetry breaking}, all eigenvalues of $\mu$ are different and the symmetry breaks from $\SU (N)$ to $\S \( \rU (1) \times \cdots \times \rU (1) \)$ at infinity, that is the bundle splits to a sum of complex line bundles, each of which are eigenbundles for both $\mu$ and $\kappa$.

\emph{Minimal symmetry breaking} occurs when there are only two eigenvalues and one of them has multiplicity $1$. In that case, the symmetry breaks from $\SU (N)$ to $\S \( \rU (N - 1) \times \rU (1) \)$. The simplest minimal symmetry breaking case is when the structure group is $\SU (3)$ and was studied in \cites{Dancer-Leese-numerical,Dancer-Leese-dynamicsSU3monopoles,Dancer-SU3monopoles,Dancer-NahmHyperKahler}; Dancer's work has been an inspiration for our desire to understand the Nahm transform for arbitrary symmetry breaking. We should note that there is not consensus about the definition for minimal symmetry breaking, and a more general version of minimal symmetry breaking is that there is only two eigenvalues, without restrictions on the multiplicity. For a thorough discussion on symmetry breaking, see \cite[Sec.~2.3]{Lang-Thesis}

In \cite{MS03}, Murray and Singer made four conjectures that are key to understanding monopoles with nonmaximal symmetry breaking and their moduli spaces. This work and its companion paper \cite{charbonneau_nahm_2023} are partially motivated by these conjectures. One should note that one of these conjectures was recently proven by Mendizabal in \cite{Mendizabal-HK}.

The heuristic of the Nahm transform is clear and well-known, originating in \cite{nahmEqn}; see also expository work including \cite{jardimSurvey}, \cite{benoitthesis}*{Chapter~2}, and \cite{FM-Nahm-book}*{Chapter~5}. In principle, one should find a direct connection between the poles of the Nahm data and the asymptotic behavior of the monopole. This analysis was circumvented in the work of Hurtubise--Murray \cite{hurtubise_construction_1989} in the case of maximal symmetry breaking for the classical groups. Their proof uses algebraic geometry and is summarized in the following diagram: 
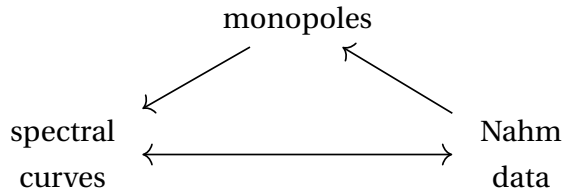
\begin{equation}
\begin{tikzcd}
   & \text{monopoles} \arrow[ld] & \\
   \begin{array}{c} \text{spectral} \\ \text{curves} \end{array} \arrow[rr, leftrightarrow] & & \begin{array}{c} \textnormal{Nahm} \\ \textnormal{data} \end{array} \arrow[ul]
\end{tikzcd}
\end{equation}
In this paper we introduce a new class of Nahm data and use them to construct monopoles with arbitrary symmetry breaking, without passing through the spectral curve. Our methods and proofs are completely analytic. The full story is recounted in the forthcoming companion paper \cite{charbonneau_nahm_2023}, where we study the reverse direction of the transform for such monopoles.

\smallskip

Without further ado, let us dive into a cartoon picture of our result, which is expanded in full in the remainder of this paper. In the maximal symmetry breaking case, all the eigenvalues of $\mu$ are different, but in the general case they need not be so. To account for multiplicities we must adapt our notation and have distinct eigenvalues $i \lambda_1, i \lambda_2, \ldots, i \lambda_n$ of $\mu$, with $i \lambda_a$ of multiplicity $r_a$, and ordered so that $ \lambda_1 < \lambda_2 < \cdots < \lambda_n$. To each eigenvalue $i \lambda_a$ corresponds an eigen-decomposition of $\kappa$ restricted to that eigenspace, with a certain number of positive eigenvalues $k_{a, 1}^+, k_{a, 2}^+, \ldots, k_{a, r_a^+}^+$, a certain number of negative eigenvalues $- k_{a, 1}^-, - k_{a, 2}^-, \ldots, - k_{a, r_a^-}^-$, and with $r_a^0$ many zero eigenvalues. This data determines nonnegative integers $m_a$. We then look for solutions to \emph{Nahm's equation}
\begin{equation}
   \dot{T}_1 (t) = \left[ T_2 (t), T_3 (t) \right], \quad \dot{T}_2 (t) = \left[ T_3 (t), T_1 (t) \right], \quad \mbox{and} \quad \dot{T}_3 (t) = \left[ T_1 (t), T_2 (t) \right],
\end{equation}
on the intervals $\( \lambda_a, \lambda_{a + 1} \)$, with $T_\alpha$ taking values in $\u \( m_a \)$. As $t$ approaches either end of $\( \lambda_a, \lambda_{a + 1} \)$, the Nahm matrices $T_\alpha$ develop poles whose residue form representations of $\SU (2)$ that decomposes into irreducible components of dimensions given by the eigenvalues of $\kappa$. The precise behavior is explained in detail in \Cref{sec:nahm_data}.

We then use such solutions to Nahm's equation to construct compatible Sobolev spaces together with a family of nondegenerate, Dirac-type, Fredholm operators labeled by vectors $x \in \rl^3$, called the \emph{Dirac--Nahm operators}. These operators are direct generalizations of ones used in \cites{nahmEqn,hitchinMonopoles,nakajima}. Using this family we construct a monopole with arbitrary, prescribed symmetry breaking.

\smallskip

In the version of Hurtubise and Murray only one of $r_a^-, r_a^0, r_a^+$ is nonzero at each $\lambda_a$, as they study only the case of maximal symmetry breaking. Hence in that setup one either has a residue on the left or on the right (but not on both) and the corresponding representation is always irreducible. Furthermore, $r_a^0$ can only be zero or one, and in the case where $r_a^0 = 1$, the Nahm data must have a discontinuity of rank $1$. In the current scenario of arbitrary symmetry breaking, residues can occur both on left and right of the $\lambda_a$,  discontinuities of the data arises in rank $r_a^0$ which can be bigger than $1$, and all the behaviors can occur at once at any of the eigenvalues.

The minimal symmetry breaking case with structure group $\SU (3)$ was studied by Dancer and Leese in \cites{Dancer-Leese-numerical,Dancer-Leese-dynamicsSU3monopoles,Dancer-SU3monopoles,Dancer-NahmHyperKahler}, and Houghton--Weinberg \cite{Houghton-Weinberg-PRD}. Houghton--Weinberg also studied (again in \cite{Houghton-Weinberg-PRD}) the case of symmetry breaking from $\SU (N)$ to $\S \( \rU (1) \times \rU (N - 2) \times \rU (1) \)$ where the middle eigenvalues of $\mu$ are equal, the eigenvalues of $\kappa$ corresponding to the repeated eigenvalues of $\mu$ are all $i$. These setups simplify considerably the analysis and hides the actual general picture we have uncovered. Indeed, the irreducible representation of dimension $1$ of $\SU (2)$ is trivial, and thus the Nahm data has no pole at the corresponding eigenvalue.

\smallskip

Nahm data occur in the description of other gauge theoretic objects as well. Heuristically, the Nahm transform maps instantons on $\S^1 \times \rl^3$ to some sort of Nahm data on $\S^1$. This correspondence was first studied in full in \cite{benoitjacques2} for structure group $\SU (2)$, following groundwork of \cites{nye,nyesinger}. This work was extended in \cite{Takayama-NahmEqm-Quivers} to arbitrary rank.

More generally, Nahm data occur in bow diagrams, a blend of intervals (on which the Nahm data live) and quivers (on which ADHM-like data live); this framework was introduced by Cherkis in \cite{Cherkis-InstantonsGravitons}. That bow data correspond to instantons on multi-Taub--NUT spaces (which are in the next more complicated ALF 4-manifolds after $\S^1 \times \rl^3$) through the Nahm transform is the recent result of \cites{Cherkis-Larrain-Stern-1,Cherkis-Larrain-Stern-2,Cherkis-Larrain-Stern-3}, and a monad description was the topic of \cites{Cherkis-Hurtubise-monads-instantons-bows,Cherkis-Hurtubise-bows-classical}.

In all of those cases, the representations occurring as residues at the end of the Nahm intervals are irreducible. Our work opens the possibilities of exploring those cases with reducible representations, producing instantons whose behavior at infinity is more general.

\smallskip

Further on this topic, Cherkis predicts in \cite{Cherkis-octonions-monopoles-knots} a potential octonionic version of the Nahm transform. As a start, He in \cite{He-octonionic-Nahm} studies at length the moduli space of solutions to the octonionic Nahm's equation, a version obtained from dimensional reduction of the $\Spin (7)$ instanton equation instead of the anti-self-dual equation that yields the standard Nahm's equation studied in this paper. The residues of the octonionic Nahm data studied by He at the end of intervals are irreducible. Our work suggests a natural extension, and should be considered in Cherkis' program.

\smallskip

Throughout the many decades of study of monopoles, there has been a plethora of symmetric monopoles produced and studied: for structure group $\SU (2)$, for higher rank with minimal symmetry breaking, on the hyperbolic space. The ideas of this paper have already been used in \cite{Charbonneau-construction-monopoles-symmetries} to produce novel examples of spherically and axially symmetric monopoles and, in fact, characterize all such solutions under a technical assumption that has been removed recently in \cite{Lang-hyperbolic-with-symmetries}.

\medskip

\subsection*{Organization of the paper:} In \Cref{sec:BPS}, we briefly outline the gauge theoretic background for BPS monopoles on $\rl^3$. In \Cref{sec:nahm_data}, we introduce the Nahm data needed for our construction. In \Cref{sec:fredholm}, we study the Fredholm theory of the Dirac--Nahm operator and use it to construct a monopole on $\rl^3$. In \Cref{sec:main} we prove that the monopoles constructed in \Cref{sec:fredholm} have the desired symmetry breaking. Finally, in \Cref{app:SO/Sp}, we show how to use our results to construct monopoles with real orthogonal or compact symplectic structure groups and in \Cref{app:low_rank} we give examples of these novel Nahm data corresponding to low rank BPS monopoles with nonmaximal symmetry breaking.

\bigskip

\section{Monopoles}
\label{sec:BPS}

Let $E$ be a smooth, rank $N$, Hermitian vector bundle over $\rl^3$. A pair, $\( \nabla, \Phi \)$, of an $\SU (N)$ connection, $\nabla$, on $E$ and a traceless, skew-adjoint endomorphism, $\Phi$, of $E$, is a \emph{monopole} if it satisfies the \emph{(Bogomolny) monopole equation}
\begin{equation}
   F_\nabla = \ast \nabla \Phi, \label{eq:mono1}
\end{equation}
and has finite \emph{Yang--Mills--Higgs energy}
\begin{equation}
   \cE_\YMH \( \nabla, \Phi \) \eqdef \frac{1}{4 \pi} \int\limits_{\rl^3} \( \left| F_\nabla \right|^2 + \left| \nabla \Phi \right|^2 \) \vol. \label{eq:YMH_energy}
\end{equation}
Note that if \cref{eq:mono1} holds then the finiteness of $\cE_\YMH \( \nabla, \Phi \)$ is equivalent to
\begin{equation}
   |\nabla \Phi| \in L^2 \( \rl^3 \). \label{eq:mono_L2}
\end{equation}
We remark that the monopole \cref{eq:mono1} imply
\begin{equation}
   \nabla^* \nabla \Phi = 0. \label{eq:Phi_harmonic}
\end{equation}

\smallskip

\begin{remark}
   \label{remark:U(N)}
   One can, of course, consider $\rU (N)$-monopoles as well. In that case, if $\( \nabla, \Phi \)$ is a solution to \cref{eq:Phi_harmonic,eq:mono_L2}, then the pair $\( \nabla^{\det \( E \)}, \tr \( \Phi \) \)$ is a finite energy $\rU (1)$-monopole on $\det \( E \) \eqdef \bigwedge^N E$. Since then $\tr \( \Phi \) : \rl^3 \rightarrow i \rl$ is a bounded and harmonic function, it must be constant. Moreover $\nabla^{\det \( E \)}$ is gauge-equivalent to the product connection, and $\nabla$ induces a reduction of the structure group to $\SU (N)$. In particular, the pair $\( \nabla, \Phi - \tfrac{\tr \( \Phi \)}{N} \id \)$ is gauge equivalent to a finite energy $\SU (N)$-monopole.
\end{remark}

We impose asymptotic conditions on $\( \nabla, \Phi \)$, following \cite{MS03}: Let $r$ be the euclidean distance (radial coordinate) from the origin on $\rl^3$. Let $\S_\infty^2$ the ``sphere at infinity,'' that is the space of oriented lines in $\rl^3$ that pass through the origin, and equip $\S_\infty^2$ with the round metric of radius one, and let $\pi \colon \( \rl^3 - \left\{ \textnormal{o} \right\} \) \rightarrow \S_\infty^2$ be the obvious map. Let $E^\infty \rightarrow \S_\infty^2$ be a (necessarily topologically trivial) Hermitian vector bundle with structure group $\SU (N)$. We then assume that there exist
\begin{enumerate}

   \item sections $\mu, \kappa$ of $\su \( E^\infty \)$,

   \item an $\SU (N)$ connection $\nabla^\infty$ on $E^\infty$,

   \item and an isomorphism of $E|_{\rl^3 - \left\{ \textnormal{o} \right\}}$ and $\pi^* \( E^\infty \)$,

\end{enumerate}
such that, under the above isomorphism, we have on $\rl^3 - \left\{ \textnormal{o} \right\}$ that
\begin{subequations}
\begin{align}
   \nabla      &= \del_r \otimes \rd r + \nabla^\infty + \tfrac{1}{r^2} \cR_1, \label[cond]{eq:mono2} \\
   \Phi        &= \mu - \tfrac{1}{2r} \kappa + \tfrac{1}{r^2} \cR_2, \label[cond]{eq:mono3} \\
   \nabla \Phi &= \tfrac{1}{2 r^2} \kappa \otimes \rd r + \tfrac{1}{r^3} \cR_3, \label[cond]{eq:mono4}
\end{align}
\end{subequations}
for some $\cR_1$, $\cR_2$, and $\cR_3$ each satisfying
\begin{equation}
   \limsup\limits_{r \rightarrow \infty} \( \| \cR_i \|_{L^\infty \( \S_r^2 \)} + \| \nabla \cR_i \|_{L^\infty \( \S_r^2 \)} \) < \infty. \label[cond]{eq:mono5}
\end{equation}
Note that the monopole \cref{eq:mono1,eq:mono2,eq:mono3,eq:mono4,eq:mono5} imply
\begin{subequations}
\begin{align}
   \nabla^\infty \mu    &= 0, \label{eq:mu_parallel} \\
   \nabla^\infty \kappa &= 0, \label{eq:kappa_parallel} \\
   F_{\nabla^\infty}    &= \tfrac{1}{2} \kappa \otimes \vol_{\S_\infty^2}. \label{eq:curvature_at_infinity}
\end{align}
\end{subequations}
Thus
\begin{equation}
   \left[ \mu, \kappa \right] = \left[ \mu, 2 \ast_{\S_\infty^2} \( \tfrac{1}{2} \kappa \otimes \vol_{\S_\infty^2} \) \right] = - 2 \ast_{\S_\infty^2} \left[ F_{\nabla^\infty}, \mu \right] = - 2 \ast_{\S_\infty^2} \rd_{\nabla^\infty}^2 \mu = 0, \label{eq:mu_kappa_commutator}
\end{equation}
and since $|F_\nabla| = |\nabla \Phi|$, we get
\begin{align}
   \cE_\YMH \( \nabla, \Phi \)   &= \tfrac{2}{4 \pi} \int\limits_{\rl^3} \left| \nabla \Phi \right|^2 \vol \\
      &= \tfrac{1}{2 \pi} \langle \nabla^* \nabla \Phi | \Phi \rangle_{L^2 \( \rl^3 \)} + \tfrac{1}{2 \pi} \lim\limits_{r \rightarrow \infty} \langle \nabla_{\del_r} \Phi | \Phi \rangle_{L^2 \( \S_r^2 \)} \\
      &= - \tr \( \mu \kappa \). \label{eq:energymuk}
\end{align}
By \cref{eq:kappa_parallel,eq:curvature_at_infinity}, $\nabla^\infty$ is a Yang--Mills connection. Thus, by \cite{grothendieck_classification_1957}*{Theorem~2.1}, $\( E^\infty, \nabla^\infty \)$ is reducible to a direct sum of Hermitian line bundles in a way that on each line bundle the induced connection is also a Yang--Mills connection, that is, its curvature is $\tfrac{i}{2} k \otimes \vol_{\S_\infty^2}$ for some integer $k \in \Z$. The sum of these integers is necessarily zero.

There are further necessary conditions that the spectra of $\mu$ and $\kappa$ ought to satisfy (see \cite{MS03}*{Section~4}), but we leave the discussion of those to the next section.

\begin{remark}
   Given two bundles, $E_1$ and $E_2$, and monopoles  $\( \nabla_1, \Phi_1 \)$ and $\( \nabla_2, \Phi_2 \)$ on them, one can construct a new monopole on $E = E_1 \oplus E_2$ via
   \begin{equation}
      \( \nabla, \Phi \) = \( \nabla_1 \oplus \nabla_2, \Phi_1 \oplus \Phi_2 \).
   \end{equation}
   More generally, if  $\( \nabla_1, \Phi_1 \)$ and $\( \nabla_2, \Phi_2 \)$ have unitary (but not necessarily special unitary) structure groups, say $\rU \( N_1 \)$ and $\rU \( N_2 \)$, then for any $\lambda_1, \lambda_2 \in \rl$, the field configuration
   \begin{equation}
      \( \nabla, \Phi \) = \( \nabla_1 \oplus \nabla_2, \( \Phi_1 - i \lambda_1 \id_{E_1} \) \oplus \( \Phi_2 - i \lambda_2 \id_{E_2} \) \),
   \end{equation}
   is also a $\rU \( N_1 + N_2 \)$-monopole. Moreover, as in \Cref{remark:U(N)}, $\tr \( \Phi_1 \)$ and $\tr \( \Phi_2 \)$ are constants, so setting $\lambda_1 = - \tfrac{i}{N_1} \tr \( \Phi_1 \)$ and $\lambda_2 = - \tfrac{i}{N_2} \tr \( \Phi_2 \)$ yields that
   \begin{equation}
      \( \nabla, \Phi \) = \( \nabla_1 \oplus \nabla_2, \( \Phi_1 - \tfrac{1}{N_1} \tr \( \Phi_1 \) \id_{E_1} \) \oplus \( \Phi_2 - \tfrac{1}{N_2} \tr \( \Phi_2 \) \id_{E_2} \) \), \label{eq:reducible}
   \end{equation}
   is an $\SU \( N_1 + N_2 \)$-monopole. We call such a monopole \emph{reducible}.

   Note that if, say, $N_2 = 1$, then, as in \Cref{remark:U(N)}, we get that $\nabla_2$ is (gauge equivalent to) the product connection and $\Phi_2 - \tfrac{1}{N_2} \tr \( \Phi_2 \) \id_{E_2} = 0$. In this case we call $\( \nabla_2, \Phi_2 \)$ a \emph{flat factor}. Conversely, one can always add flat factors to a monopole the above way. This process creates an extra $0$ eigenvalue in $\kappa$.

   In this paper we only consider monopoles \emph{without flat factors}. Note that such monopoles can still be reducible.
\end{remark}

\bigskip

\section{Nahm data with nonmaximal symmetry breaking}
\label{sec:nahm_data}

In this section, we construct $\SU (N)$ monopoles with prescribed symmetry breaking. By \cref{eq:mono3}, the asymptotic behavior of the Higgs field is determined by a pair of elements in $\Gamma \( \su \( E^\infty \) \)$, $\mu$ and $\kappa$. By \cref{eq:mu_parallel,eq:kappa_parallel}, these elements have constant spectrum and by \cref{eq:mu_kappa_commutator}, they can be simultaneously diagonalized. Finally, by \cref{eq:curvature_at_infinity}, $i \kappa$ has only integer eigenvalues. For each eigenvalue $i \lambda \in \Spec \( \mu \)$, we then have, in some gauge and for some positive integers $k_a^\pm \geqslant k_{a + 1}^\pm$, that 
\begin{equation}
   \kappa|_{\ker \( \mu - i \lambda \id \)} = \diag \( i k_1^+, i k_2^+, \ldots, 0, \ldots, - i k_2^-, - i k_1^- \).
\end{equation}
(The number of positive/zero/negative elements can be zero.)

\smallskip

With that in mind, we define the \emph{type of symmetry breaking} as follows.

\begin{definition}
\label{definition:type}
   Let $n, N \in \Z_+$ and $n \leqslant N$. A \emph{symmetry breaking type of  size $n$ and rank $N$}, $\bT = \( \underline{\lambda}, \underline{r}, \underline{k} \)$, consists of a triple of $n$-tuples
   \begin{align}
      \textnormal{eigenvalues:}     \ \ \underline{\lambda}  &= \( \lambda_1, \lambda_2, \ldots, \lambda_n \) \in \rl^n, \\
      \textnormal{ranks:}           \ \ \underline{r}        &= \( r_1, r_2, \ldots, r_n \) \in \Z_+^n, \\
      \textnormal{Chern numbers:}   \ \ \underline{k}        &= \( \underline{k}_1, \underline{k}_2, \ldots, \underline{k}_n \) \in \Z^{r_1} \times \Z^{r_2} \times \cdots \times \Z^{r_n},
   \end{align}
   satisfying conditions enumerated below. These conditions are best formulated using additional notation. We label the entries in the $\underline{k}_a$ as
      \begin{equation}
         \underline{k}_a = \( k_{a, 1}^+, k_{a, 2}^+, \ldots, k_{a, r_a^+}^+, 0, \ldots, 0, - k_{a, r_a^-}^-,  \ldots - k_{a, 2}^-, - k_{a, 1}^- \),
      \end{equation}
   with $k^\pm_{a, 1} \geqslant k^\pm_{a, 2} \geqslant \ldots \geqslant k^\pm_{a, r_a^\pm} > 0$, and let
   \begin{equation}
      r_a^0 \eqdef r_a - \( r_a^+ + r_a^- \),
   \end{equation}
   be the number of zeros in $\underline{k}_a$. For each $a \in \{ 1, 2, \ldots, n \}$, we define
         \begin{align}
            k_a^\pm  &\eqdef \sum\limits_{b = 1}^{r_a^\pm} k_{a, b}^\pm, \\
            k_a      &\eqdef k_a^+ - k_a^-, \\
            m_a      &\eqdef - \sum\limits_{b = 1}^a k_b.
         \end{align}
   Finally, let $m_0 \eqdef 0$ and $r_0 \eqdef 0$. The conditions $\underline{\lambda}, \underline{r}, \underline{k}$ satisfy are
   \begin{enumerate}

      \item $\sum\limits_{a = 1}^n r_a = N$,

      \item $\sum\limits_{a = 1}^n \lambda_a r_a = 0$,

      \item $\lambda_1 < \lambda_2 < \cdots < \lambda_n$,

      \item  $r_1 = r_1^-$ and $r_n = r_n^+$,

      \item for each $a \in \{ 1, 2, \ldots, n - 1 \}$, $m_a \geqslant \max \( \left\{ \ r_a^0 + k_a^-, r_{a + 1}^0 + k_{a + 1}^+ \ \right\} \) \geqslant 0$, and $m_n = 0$.

   \end{enumerate}

\end{definition}

\smallskip

\begin{remark}
   The interpretation of $\bT = \( \underline{\lambda}, \underline{r}, \underline{k} \)$ is that for all $a \in \{ 1, 2, \ldots, n \}$:
   \begin{enumerate}
      
      \item The imaginary numbers $i \lambda_a$ are the distinct eigenvalues of the Higgs field at infinity.
      
      \item The nonnegative integers $r_a^+, r_a^-$, and $r_a^0$ are the number of line bundles with positive, negative, and zero Chern number in the holomorphic decomposition of the bundle corresponding to the eigenvalue $i \lambda_a$.
      
      \item For all $b \in \{ 1, 2, \ldots, r_a^\pm \}$ the integer $\pm k_{a, b}^\pm$ is the nonzero Chern number of a line bundle in the holomorphic decomposition of the bundle $E^\infty$.

   \end{enumerate}
   \Cref{skyscraper1}, a \emph{skyscraper} diagram, illustrates how these various quantities appear on the Nahm side; cf. \Cref{definition:nahm_data}.
   
   \begin{figure}[h!]
      \begin{tikzpicture}[scale=1]
         \draw[->, very thick] (-2,0) -- (8,0); 
         \draw[thick] (0,0) -- (0,4);  
         \draw[thick] (6,0) -- (6,4);  
         \draw[thick] (0,4) -- (6,4);  
         \draw[thick] (0,3) -- (-1.5,3); 
         \draw[thick] (6,4) -- (7.5,4); 
         \draw (0,2) -- +(-1,0); 
         \draw (0,2)++(-0.5,0.5) node {$k_a^+$};
         \draw[->] (0,2)++(-0.5,0.7) -- +(0,0.3);
         \draw[->] (0,2)++(-0.5,0.3) -- +(0,-0.3);
         \draw (0,2) -- +(1,0); 
         \draw (0,1.5)+(0.5,1.5) node {$k_a^-$};
         \draw[->] (0,1.5)++(0.5,1.5)++(0,0.2) -- +(0,0.8);
         \draw[->] (0,1.5)++(0.5,1.5)++(0,-0.2) -- +(0,-0.8);
         \draw (3,2) node {$m_a$};
         \draw[->] (3,2)+(0,0.2) -- +(0,2);
         \draw[->] (3,2)+(0,-0.2) -- +(0,-2);
         \draw[dashed] (-0.5,1) -- (0.5,1);
         \draw (0.3,0.5) node {$r_a^0$};
         \draw[->] (0.3,0.7) -- (0.3,1);
         \draw[->] (0.3,0.3) -- (0.3,0);
         \draw (0,0) -- (0,-0.2);
         \draw (0,-0.4) node {$\lambda_a$};
         \draw (6,0) -- (6,-0.2);
         \draw (6,-0.4) node {$\lambda_{a+1}$};
      \end{tikzpicture}
      \caption{Skyscraper diagram illustrating the various dimension at play in Nahm data in view of the constants introduced in \Cref{definition:type}}\label{skyscraper1}
   \end{figure}

\end{remark}

\smallskip

Next we define the framing of a symmetry breaking type which can be viewed as an analogue of framed gauge equivalence classes for monopoles.

\begin{definition}\label{definition:framing}
   A \emph{framing of a symmetry breaking type $\bT$}, is an $\( n - 1 \)$-tuple $\cF = \( C_a \)_{a = 2}^{n - 1}$ such that, for all $a \in \{ 2, \ldots, n - 1 \}$, we have that $C_a \in \Hom \( \cx^{m_{a - 1}}, \cx^{m_a} \)$, and if
   \begin{equation}
      V_{a - 1}^+ \eqdef \ker \( C_a \) \subseteq \cx^{m_{a - 1}} \quad \& \quad V_a^- \eqdef \( \image \( C_a \) \)^\perp \subseteq \cx^{m_a},
   \end{equation}
   then $\dim_\cx \( V_a^- \) = k_a^-$ and $\dim_\cx \( V_a^+ \) = k_{a + 1}^+$ and the maps $C_a\big|_{\( V_{a - 1}^+ \)^\perp} \colon \( V_{a - 1}^+ \)^\perp \rightarrow \( V_a^- \)^\perp$ are unitary isomorphisms.
\end{definition}

\smallskip

\begin{remark}
   \label{remark:end_points}
   Note that $V_1^- = \cx^{k_1^-} = \cx^{m_1}$ and $V_{n - 1}^+ = \cx^{k_n^+} = \cx^{m_{n - 1}}$.
\end{remark}

\smallskip

Let $\bH = \cx \oplus \cx j$ be the field of quaternions and $\sigma_1, \sigma_2, \sigma_3$ be the usual Pauli operators, that is, for all $q = z_1 + z_2 j \in \bH$, we have
\begin{align}
   \sigma_1 \( q \) \eqdef q i = i z_1 - i z_2 j, \quad \sigma_2 \( q \) \eqdef q j = - z_2 + z_1 j, \quad \& \quad \sigma_3 \( q \) \eqdef q k = q i j = i z_2 - i z_1 j.
\end{align}
Note that for all $\alpha, \beta \in \{ 1, 2, 3 \}$, we have $\sigma_\alpha \circ \sigma_\beta = - \delta_{\alpha, \beta} + \sum_{\gamma = 1}^3 \epsilon_{\alpha, \beta, \gamma} \sigma_\gamma$. In particular, $\left[ \sigma_\alpha, \sigma_\beta \right] = \sum_{\gamma = 1}^3 \epsilon_{\alpha, \beta, \gamma} \sigma_\gamma$, thus we can identify $\su (2)$ with $\spn_\rl \( \{ \sigma_1, \sigma_2, \sigma_3 \} \)$. With this identification, $\bH$ becomes the defining representation of $\su (2)$.

We consider $\bH$ as a complex \emph{left} module, that is for all $z \in \cx$ and $q = z_1 + z_2 j \in \bH$, we have
\begin{equation}
   z \cdot q \eqdef z z_1 + z z_2 j.
\end{equation}
Finally, let
\begin{equation}
   J \colon \bH \rightarrow \bH ; \ni z_1 + z_2 j \mapsto J \( z_1 + z_2 j \) \eqdef j \( z_1 + z_2 j \) = - \overline{z}_2 + \overline{z}_1 j, \label{eq:quaternionic_str}
\end{equation}
Let $R_k$ denote the irreducible representation of $\su(2)$ of dimension $k$.

\begin{definition}
\label{definition:nahm_data}
   A \emph{Nahm data with symmetry breaking type $\bT$ and framing $\cF$} is a collection, $\cN = \( \rd^{(a)}, T^{(a)} \)_{a = 1}^{n - 1}$, such that:
   \begin{enumerate}

      \item For all $a \in \{ 1, 2, \ldots, n - 1 \}$, $\rd^{(a)}$ is a unitary connection of the bundle $\( \lambda_a, \lambda_{a + 1} \) \times \cx^{m_a}$, $\dom \( T^{(a)} \) = \( \lambda_a, \lambda_{a + 1} \)$, and
      \begin{equation}
         T^{(a)} \colon \( \lambda_a, \lambda_{a + 1} \) \rightarrow \u \( m_a \)^{\oplus 3}; \ t \mapsto \( T_1^{(a)} (t), T_2^{(a)} (t), T_3^{(a)} (t) \)
      \end{equation}
      is an analytic function.

      \item For all $a \in \{ 1, 2, \ldots, n - 1 \}$, let $\rd^{(a), 0}$ be the product connection on $\( \lambda_a, \lambda_{a + 1} \) \times \cx^{m_a}$. Then
      \begin{equation}
         T_0^{(a)} \colon \( \lambda_a, \lambda_{a + 1} \) \rightarrow \u \( m_a \); \ t \mapsto \rd_{\del_t}^{(a)} - \rd_{\del_t}^{(a), 0} \label{eq:T_0_def}
      \end{equation}
      extends smoothly over $\left[ \lambda_a, \lambda_{a + 1} \right]$.

      \item $\cN$ solves \emph{Nahm's equation}, that is, for all $a \in \{ 1, 2, \ldots, n - 1 \}$, we have
      \begin{equation}
         \rd_{\del_t}^{(a)} T_1^{(a)} = \left[ T_2^{(a)}, T_3^{(a)} \right], \quad \rd_{\del_t}^{(a)} T_2^{(a)} = \left[ T_3^{(a)}, T_1^{(a)} \right], \quad \mbox{and} \quad \rd_{\del_t}^{(a)} T_3^{(a)} = \left[ T_1^{(a)}, T_2^{(a)} \right]. \label{eq:nahm}
      \end{equation}

      \item For all $a \in \{ 1, 2, \ldots, n - 1 \}$, the decompositions in \Cref{definition:framing} induce embeddings:
      \begin{equation}
         \u \( k_a^- \) \oplus \u \( m_a - k_a^- \) \hookrightarrow \u \( m_a \) \quad \& \quad \u \( k_{a + 1}^+ \) \oplus \u \( m_a - k_{a + 1}^+ \) \hookrightarrow \u \( m_a \),
      \end{equation}
      Under these decompositions, there are elements $\rho_{\alpha, a}^\pm, \mathfrak{T}_{\alpha, a}^\pm \in \u \( k_a^\pm \)$ and $\tau_{\alpha, a}^\pm \in \u \( m_a - k_a^- \)$, such that for $0 < \epsilon \ll 1$ we have that
      \begin{subequations}
      \begin{align}
         T_\alpha^{(a)} \( \lambda_a + \epsilon \)         &=  \begin{pmatrix}
                                                                  - \tfrac{1}{2 \epsilon} \rho_{\alpha, a}^- + \mathfrak{T}_{\alpha, a}^- & 0 \\
                                                                  0 & \tau_{\alpha, a}^-
                                                               \end{pmatrix} + O \( \epsilon \), \label{eq:nahm_pole_+} \\
         T_\alpha^{(a)} \( \lambda_{a + 1} - \epsilon \)   &=  \begin{pmatrix}
                                                                  \tfrac{1}{2 \epsilon} \rho_{\alpha, a + 1}^+ + \mathfrak{T}_{\alpha, a + 1}^+ & 0 \\
                                                                  0 & \tau_{\alpha, a + 1}^+
                                                               \end{pmatrix} + O \( \epsilon \). \label{eq:nahm_pole_-}
      \end{align}
      \end{subequations}
      Note that $m_a - k_a^- = m_{a - 1} - k_a^+$. Furthermore, the maps
      \begin{equation}
         \widehat{\rho}_a^\pm \colon \su (2) \rightarrow \u \( k_a^\pm \); \ x_1 \sigma_1 + x_2 \sigma_2 + x_3 \sigma_3 \mapsto x_1 \rho_{1, a}^\pm + x_2 \rho_{2, a}^\pm + x_3 \rho_{3, a}^\pm,
      \end{equation}
      are Lie algebra homomorphisms that decompose to a direct sum of irreducible $\su (2)$-representations as 
      \begin{equation}
         \widehat{\rho}_a^\pm \cong \bigoplus_{b = 1}^{r_a^\pm} R_{k_{a, b}^\pm}.  \label{eq:reduction}
      \end{equation}
      Furthermore, $\mathfrak{T}_{\alpha, a}^\pm$ is block-diagonal with respect to the decomposition in \cref{eq:reduction}.

      \item For $\alpha \in \{ 1, 2, 3 \}$ and $a \in \{ 2, \ldots, n - 1 \}$, let $\tau_{\alpha, a}^\pm$ be as in \cref{eq:nahm_pole_+,eq:nahm_pole_-} and let us define the self-adjoint operators
      \begin{equation}
         \fC_a \eqdef \sum\limits_{\alpha = 1}^3 \( C_a^* \circ \tau_{\alpha, a }^- \circ C_a - \tau_{\alpha, a-1}^+ \) \otimes \sigma_\alpha \colon \( V_{a - 1}^+ \)^\perp \otimes \bH \rightarrow \( V_{a - 1}^+ \)^\perp \otimes \bH, \label{eq:connecting_maps}
      \end{equation}
      Then there are nonzero and orthogonal elements $\upxi_{a, 1}, \upxi_{a, 2}, \ldots, \upxi_{a, r_a^0} \in \( V_{a - 1}^+ \)^\perp \otimes \bH$ such that
      \begin{equation}
         \fC_a = \sum\limits_{b = 1}^{r_a^0} \upxi_{a, b} \otimes \upxi_{a, b}^\flat, \label{eq:C_a_def}
      \end{equation}
      and let
      \begin{equation}
         \fB_a \eqdef \sqrt{\fC_a} = \sum\limits_{b = 1}^{r_a^0} \tfrac{1}{\left| \upxi_{a, b} \right|} \upxi_{a, b} \otimes \upxi_{a, b}^\flat. \label{eq:B_a_def}
      \end{equation}
      For all $a \in \{ 2, \ldots, n - 1 \}$, let
      \begin{align}
         X_a^\pm     &\eqdef V_a^\pm \otimes \bH, \\
         Y_{a - 1}^+ &\eqdef \spn_\cx \( \left\{ \ \upxi_{a, 1}, \upxi_{a, 2}, \ldots, \upxi_{a, r_a^0} \ \right\} \), \\
         Y_a^-       &\eqdef \( C_a \otimes \id_\bH \) \( Y_{a - 1}^+ \), \\
         Z_a^\pm     &\eqdef \( X_a^\pm \oplus Y_a^\pm \)^\perp.
      \end{align}
      In particular, for all $a \in \{ 1, 2, \ldots, n - 1 \}$, we have the orthogonal decomposition $\cx^{m_a} \otimes \bH = X_a^\pm \oplus Y_a^\pm \oplus Z_a^\pm$ and when $r_a^0$ is nonzero, then $\fC_a$ is positive definite exactly on $Y_a^-$. Furthermore, by \Cref{remark:end_points}, $Y_1^- = Z_1^-$ and $Y_{n - 1}^+ = Z_{n - 1}^+$ are the trivial subspaces of $\cx^{m_1} \otimes \bH$ and $\cx^{m_{n - 1}} \otimes \bH$, respectively.
   \end{enumerate}
\end{definition}

The various subspaces defined in \Cref{definition:framing} and \Cref{definition:nahm_data} are illustrated suggestively in skyscraper diagrams in \Cref{figure:skyscrapers}. 

\begin{figure}[h]
   \begin{tikzpicture}[scale=0.7]
      \draw (-1.6, 4.5) node {${\cx}^{m_{a - 1}}$};
      \draw (1.5, 4.5) node {${\cx}^{m_a}$};
      \draw[->] (-1,4.5) -- (1,4.5);
      \draw (0,4.8) node {$C_a$};
      \draw[->, very thick] (-3.5,0) -- (4,0); 
      \draw[thick] (0,0) -- (0,4);  
      \draw[thick] (0,4) -- (-3,4); 
      \draw[thick] (0,3) -- (3.5,3); 
      \draw (0,1) -- +(-2,0); 
      \draw (0,1)++(-2,1) node {$V_{a - 1}^+ = \ker \( C_a \)$};
      \draw[->] (0,1)++(-1.5,1)+(0,0.3) -- (-1.5,4);
      \draw[->] (0,1)++(-1.5,1)+(0,-0.3) -- (-1.5,1);
      \draw (0,1) -- +(2,0); 
      \draw (0,1)++(2.8,1) node {$V_a^- = \( \image \( C_a \) \)^\perp$};
      \draw[->] (0,1)++(1.5,1)+(0,0.3) -- (1.5,3);
      \draw[->] (0,1)++(1.5,1)+(0,-0.3) -- (1.5,1);
      \draw (0,0) -- (0,-0.2);
      \draw (0,-0.4) node {$\lambda_a$};
   \end{tikzpicture}
   \quad\quad
   \begin{tikzpicture}[scale=0.7]
      \draw (-2.9, 2.5) node {$\cx^{m_{a - 1}}\otimes \bH$};
      \draw (2.9,2.5) node {$\cx^{m_a}\otimes \bH$};
      \draw[->, very thick] (-3.5,0) -- (4,0); 
      \draw[thick] (0,0) -- (0,5);  
      \draw[thick] (0,5) -- (-3,5); 
      \draw[thick] (0,4) -- (3.5,4); 
      \draw (0,3) -- +(-1.5,0); 
      \draw (0,3)++(-.7,1) node {$X_{a - 1}^+$};
      \draw (0,1)++(-.7,1) node {$Z_{a - 1}^+$};
      \draw[dashed] (0,1) -- +(-1.5,0);
      \draw (0,1)++(-.7,-0.5) node {$Y_{a - 1}^+$};
      \draw (0,3) -- +(1.5,0); 
      \draw (0,3)++(.5,.5) node {$X_a^-$};
      \draw (0,1)++(.5,1) node {$Z_a^-$};
      \draw[dashed] (0,1) -- +(1.5,0);
      \draw (0,1)++(.5,-0.5) node {$Y_a^-$};
      \draw (0,0) -- (0,-0.2);
      \draw (0,-0.4) node {$\lambda_a$};
   \end{tikzpicture}
   \caption{Skyscraper diagrams illustrations for \Cref{definition:framing} and \Cref{definition:nahm_data}.}\label{figure:skyscrapers}
\end{figure}

Next, let us define the appropriate gauge group and moduli space.

\begin{definition}
   Let the \emph{gauge group} be
   \begin{equation}
      \cG^\bT \eqdef \bigoplus\limits_{a = 1}^{n - 1} C^\infty \( \left[ \lambda_a, \lambda_{a + 1} \right], \rU \( m_a \) \),
   \end{equation}
   with the natural group structure. If $g = \( g_a \)_{a = 1}^{n - 1} \in \cG^\bT$ and $\cN = \( \rd^{(a)}, T^{(a)} \)_{a = 1}^{n - 1}$ is a Nahm data according to \Cref{definition:nahm_data}, with framing $\cF = \( C_a \)_{a = 2}^{n - 1}$, then let
   \begin{align}
      g \( \cF \) &\eqdef \( g_a \( \lambda_a \) \circ C_a \circ g_{a - 1} \( \lambda_a \)^{- 1} \)_{a = 2}^{n - 1}, \\
      g \( \cN \) &\eqdef \( g_a \circ \rd^{(a)} \circ g_a^{- 1}, \Ad \( g_a \) \circ T^{(a)} \)_{a = 1}^{n - 1}.
   \end{align}
   Then $g \( \cN \)$ is again a Nahm data with symmetry breaking type $\bT$ and framing $g \( \cF \)$.
\end{definition}

\smallskip

\begin{remark}
   \label{eq:unique_framing}
   Every symmetry breaking type determines a unique gauge equivalence class of framings.
\end{remark}

\smallskip

\begin{definition}
   Let $\cB^\bT$ be the space of all Nahm data with symmetry breaking type $\bT$. By the above discussion, $\cG^\bT$ acts on $\cB^\bT$. We define the \emph{moduli space of Nahm data with symmetry breaking type $\bT$} as
   \begin{equation}
      \cM^\bT \eqdef \cB^\bT \big/ \cG^\bT.
   \end{equation}
\end{definition}

\smallskip

Using \Cref{remark:gauging}, we choose the following gauge for the purposes of this paper:

\begin{definition}
   \label{definition:temporal_gauge}
   For all $a \in \{ 1, 2, \ldots, n - 1 \}$, let $T_0^{(a)}$ be as defined in \cref{eq:T_0_def}, let us pick $t_a \in \( \lambda_a, \lambda_{a + 1} \)$, and define $g_a \colon \left[ \lambda_a, \lambda_{a + 1} \right] \rightarrow \rU$ to be the unique solution to the initial value problem
   \begin{align}
   	\frac{\rd g_a}{\rd t} (t)  &= g_a(t) T_0(t), \\
	   g_a (t_a)                  &= 1,
   \end{align}
   and $g \eqdef \( g_a \)_{a = 1}^{n - 1} \in \cG^\bT$. Then note that if $\cN^\prime \eqdef g \( \cN \)$, then the connections in $\cN^\prime$ are all product connections. We say that such a Nahm data is in \emph{temporal gauge}.
\end{definition}

Since all of our constructions in this paper are gauge equivariant, we can assume, without any loss of generality, that $\cN$ is in temporal gauge, that is, for all $a \in \{ 1, 2, \ldots, n - 1 \}$, $\rd^{(a)}$ is the production connection. Hence, for any section $\chi_a$ over $\( \lambda_a, \lambda_{a + 1} \)$, we write $\dot{\chi}_a \eqdef \rd_{\del_t}^{(a)} \chi_a$.

\smallskip

We end this section with a short lemma.

\begin{lemma}
\label{lemma:isomorphisms}
   Let $\bT = \( \underline{\lambda}, \underline{r}, \underline{k} \)$ and $\bT^\prime = \( \underline{\lambda}^\prime, \underline{r}^\prime, \underline{k}^\prime \)$ be such that $\underline{k} = \underline{k}^\prime$.
   Then $\cF$ is a framing of $\bT$ exactly if it is a framing of $\bT^\prime$. Furthermore, there exist canonical isomorphisms
   \begin{subequations}
   \begin{align}
      \cG^\bT  &\cong \cG^{\bT^\prime}, \label{eq:isom_1} \\
      \cB^\bT  &\cong \cB^{\bT^\prime}, \label{eq:isom_2} \\
      \cM^\bT  &\cong \cM^{\bT^\prime}, \label{eq:isom_3}
   \end{align}
   \end{subequations}
   such that the diagram
   \begin{equation} \label[diag]{diag:iso}
   \begin{tikzcd}
      \cG^\bT \arrow[rd, hook] \arrow[r, "\cong"]  & \cG^{\bT^\prime} \arrow[rd, hook]                & \\
                                                   & \cB^\bT \arrow[r, "\cong"] \arrow[d, two heads]  & \cB^{\bT^\prime} \arrow[d, two heads] \\
                                                   & \cM^\bT \arrow[r, "\cong"]                       & \cM^{\bT^\prime}
   \end{tikzcd}
   \end{equation}
   commutes.
\end{lemma}

\begin{proof}
   Note that $\underline{k}$ already determines $\underline{r}$, the size $n$, and rank $N$ of a symmetry breaking type. Thus, $\underline{r} = \underline{r}^\prime$, $n = n^\prime$, and $N = N^\prime$. Now a framing $\cF$ that is compatible with one of the symmetry breaking types can naturally and uniquely be regarded as a framing for the other as well. Thus, we can use the ideas of \cite{Charbonneau-construction-monopoles-symmetries}*{Remark 4.10} as follows: since the spaces in \cref{eq:isom_1,eq:isom_2} are both spaces of functions, we construct isomorphisms of their domains and prove that the corresponding pullbacks are the desired isomorphisms. Note that
   \begin{equation}
      \bigcup\limits_{a = 1}^{n - 1} \left[ \lambda_a, \lambda_{a + 1} \right] = \left[ \lambda_1, \lambda_n \right] \quad \mbox{and} \quad \bigcup\limits_{a = 1}^{n - 1} \left[ \lambda_a^\prime, \lambda_{a + 1}^\prime \right] = \left[ \lambda_1^\prime, \lambda_n^\prime \right].
   \end{equation}
   For all $a \in \{ 1, 2, \ldots, n - 1 \}$, let
   \begin{equation}
      f_a \colon \left[ \lambda_a^\prime, \lambda_{a + 1}^\prime \right] \rightarrow \left[ \lambda_a, \lambda_{a + 1} \right] ; \ t \mapsto \tfrac{\lambda_{a + 1} - \lambda_a}{\lambda_{a + 1}^\prime - \lambda_a^\prime} t + \tfrac{\lambda_{a + 1}^\prime \lambda_a - \lambda_{a + 1} \lambda_a^\prime}{\lambda_{a + 1}^\prime - \lambda_a^\prime},
   \end{equation}
   and let $f_{\bT, \bT^\prime} \colon \left[ \lambda_1^\prime, \lambda_n^\prime \right] \rightarrow \left[ \lambda_1, \lambda_n \right]$ be defined so that if $t \in \left[ \lambda_a^\prime, \lambda_{a + 1}^\prime \right]$, then $f_{\bT, \bT^\prime} (t) = f_a (t)$. Note that $f_{\bT, \bT^\prime}$ is now well-defined and continuous. For all $g = \( g_a \)_{a = 1}^{n - 1} \in \cG^\bT$, let
   \begin{equation}
      f_{\bT, \bT^\prime}^* \( g \) \eqdef \( g_a \circ f_a \)_{a = 1}^{n - 1} \in \cG^{\bT^\prime}.
   \end{equation}
   This map is the desired isomorphism in \cref{eq:isom_1}. If $\cN = \( \rd^{(a)}, T^{(a)} \)_{a = 1}^{n - 1} \in \cB^\bT$, then let
   \begin{equation}
      f_{\bT, \bT^\prime}^* \( \cN \) \eqdef \( \tfrac{\lambda_{a + 1} - \lambda_a}{\lambda_{a + 1}^\prime - \lambda_a^\prime} T^{(a)} \circ \( f_a \big|_{\( \lambda_a^\prime, \lambda_{a + 1}^\prime \)} \) \)_{a = 1}^{n - 1},
   \end{equation}
   then it is easy to check the commutativity of the diagram
   \begin{equation}
   \begin{tikzcd}
      \cG^\bT \arrow[rd, hook] \arrow[r, "f_{\bT, \bT^\prime}^*"] & \cG^{\bT^\prime} \arrow[rd, hook]          & \\
                                                                  & \cB^\bT \arrow[r, "f_{\bT, \bT^\prime}^*"] & \cB^{\bT^\prime}
   \end{tikzcd}
   \end{equation}
   Finally, one can verify, with a little more work that the map
   \begin{equation}
      \cM^\bT \ni \left[ \cN \right] \mapsto \left[ f_{\bT, \bT^\prime}^* \( \cN \) \right] \in \cM^{\bT^\prime},
   \end{equation}
   is well-defined and is an isomorphism, which proves \cref{eq:isom_3}.
\end{proof}

\begin{remark}
   In \Cref{lemma:isomorphisms} ``isomorphism'' means merely a bijection, as we omit the problem of topologizing $\cB^\bT$, and hence $\cM^\bT$. Note however that $\cG^\bT$ is a finite dimensional Lie-group, thus it has a canonical smooth structure and, of course, the identity map is a diffeomorphism. Furthermore, the map $f_{\bT, \bT^\prime}^*$ constructed in \Cref{lemma:isomorphisms} is conjecturally also a diffeomorphism for the appropriate smooth structures.
\end{remark}

\bigskip

\section{The Nahm transform and Dirac--Nahm operator}
\label{sec:fredholm}

For the rest of the paper, we fix a Nahm data, $\cN$, as above, with symmetry breaking type $\bT$ and framing $\cF$.

\smallskip

\begin{definition}[Nahm's construction]
\label{definition:nahm_construction}
   Let
   \begin{equation}
      L^2 \( \bT \) \eqdef \( \bigoplus_{a = 1}^{n - 1} L^2 \( \( \lambda_a, \lambda_{a + 1} \), \cx^{m_a} \otimes \bH \) \) \oplus \( \bigoplus\limits_{b = 2}^{n - 1} Y_{b-1}^+ \).
   \end{equation}

   We say that an element
   \begin{equation}
      \Psi = \( \uppsi_a \)_{a = 1}^{n - 1} \in \bigoplus_{a = 1}^{n - 1} L_1^2 \( \( \lambda_a, \lambda_{a + 1} \), \cx^{m_a} \otimes \bH \),
   \end{equation}
   satisfies the \emph{Nahm boundary conditions} if for all $a \in \{ 1, \ldots, n - 1 \}$, we have
   \begin{subequations}
   \begin{align}
      \lim\limits_{\epsilon \rightarrow 0^+} \Pi_{X_a^+} \( \uppsi_a \( \lambda_a + \epsilon \) \)             &= 0, \label[cond]{cond:nahm1} \\
      \lim\limits_{\epsilon \rightarrow 0^+} \Pi_{X_{a + 1}^-} \( \uppsi_a \( \lambda_{a + 1} - \epsilon \) \) &= 0, \label[cond]{cond:nahm2}
   \end{align}
   and for all $a \in \{ 2, \ldots, n - 1 \}$
   \begin{equation}
      \lim\limits_{\epsilon \rightarrow 0^+} C_a \( \uppsi_{a - 1} \( \lambda_a - \epsilon \) \) = \lim\limits_{\epsilon \rightarrow 0^+} \uppsi_a \( \lambda_a + \epsilon \). \label[cond]{cond:nahm3}
   \end{equation}
   \end{subequations}
   Furthermore, not that the Nahm boundary \cref{cond:nahm1,cond:nahm2,cond:nahm3} depend on the framing $\cF$. Now let
   \begin{equation}
      L_1^2 \( \bT, \cF \) \eqdef \left\{ \ \Psi \in \bigoplus_{a = 1}^{n - 1} L_1^2 \( \( \lambda_a, \lambda_{a + 1} \), \cx^{m_a} \otimes \bH \) \ \middle| \ \Psi \textnormal{ satisfies \cref{cond:nahm1,cond:nahm2,cond:nahm3}} \ \right\}.
   \end{equation}
   The Hilbert structure of $L_1^2 \( \bT, \cF \)$ is defined by
   \begin{equation}
      \| \Psi \|_{L_1^2 \( \bT, \cF \)}^2 \eqdef \sum\limits_{a = 1}^{n - 1} \( \| \dot{\uppsi}_a \|_{L_1^2 \( \lambda_a, \lambda_{a + 1} \)}^2 + \| \uppsi \|_{L_1^2 \( \lambda_a, \lambda_{a + 1} \)}^2 \). \label{eq:H1_norm}
   \end{equation}

   For or each $x \in \rl^3$, the \emph{Dirac--Nahm operator}, $\D_x^\cN \colon L_1^2 \( \bT, \cF \) \rightarrow L^2 \( \bT \)$, is defined via
   \begin{equation}
      \D_x^\cN \( \( \uppsi_a \)_{a = 1}^{n - 1} \) \eqdef \( i \dot{\uppsi}_a + \sum\limits_{\alpha = 1}^3 \( \( i T_\alpha^{(a)} - x_\alpha \id_{\cx^{m_a}} \) \otimes \sigma_\alpha \) \uppsi_a \)_{a = 1}^{n - 1} \oplus \( \fB_a \( \lim\limits_{\epsilon \rightarrow 0^+} \uppsi_{a - 1} \( \lambda_a - \epsilon \) \) \)_{a = 2}^{n - 1}. \label{eq:DN_op}
   \end{equation}

   By \Cref{theorem:fredholm} below, $\( \D_x^\cN \)_{x \in \rl^3}$ is a smooth family of Fredholm operators of index $- N$ and trivial kernel. Hence, $E^\cN \eqdef \( E_x^\cN \eqdef \coker \( \D_x^\cN \) \)_{x \in \rl^3}$ defines a smooth Hermitian vector bundle of rank $N$ over $\rl^3$. Let $\nabla^\cN$ be the unitary connection on $E^\cN$ that is induced by the product connection on the product $L^2 \( \bT \)$-bundle over $\rl^3$. Let $\hat{t}$ be multiplication by $t$ (that is $\hat{t}(\psi)(t)=t\psi(t)$) and
   \begin{equation}
      M \colon L^2 \( \bT \) \rightarrow L^2 \( \bT \); \: \( \uppsi_a \)_{a = 1}^{n - 1} \oplus \( \upphi_b \)_{b = 2}^{n - 1} \mapsto \( i \hat{t} \( \uppsi_a \) \)_{a = 1}^{n - 1} \oplus \( i \lambda_b \upphi_b \)_{b = 2}^{n - 1}, \label{eq:M_def}
   \end{equation}
   and, for all $x \in \rl^3$, let $\Pi_{E_x^\cN}$ be the orthogonal projection onto $E_x^\cN$ and let $\Phi_x^\cN$ be the operator   
   \begin{equation}
      \Phi_x^\cN \colon E_x^\cN \rightarrow E_x^\cN; \ \Psi \mapsto  \( \Pi_{E_x^\cN} \circ M \) \( \Psi \).
   \end{equation}
\end{definition}

\smallskip

\begin{remark}
   \label{remark:gauging}
   If $g \eqdef \( g_a \)_{a = 1}^{n - 1} \in \cG^\bT$ and $\cN^\prime \eqdef g \( \cN \)$, then the assignment
   \begin{equation}
      \Psi = \( \uppsi_a \)_{a = 1}^{n - 1} \mapsto \( g_a \( \uppsi_a \) \)_{a = 1}^{n - 1},
   \end{equation}
   defines a unitary isomorphism between $L_1^2 \( \bT, \cF \)$ and $L_1^2 \( \bT, g \( \cF \) \)$. Similarly,
   \begin{equation}
      \Psi = \( \uppsi_a \)_{a = 1}^{n - 1} \oplus \( \upxi_a \)_{a = 2}^{n - 1} \mapsto \( g_a \( \uppsi_a \) \)_{a = 1}^{n - 1} \oplus \( \lim\limits_{\epsilon \rightarrow 0^+} g_{a-1} \( \lambda_a - \epsilon \) \upxi_a \)_{a = 2}^{n - 1},
   \end{equation}
   defines a unitary automorphism of $L^2 \( \bT \)$. Let us denote both by $\widehat{g}$. Now for all $x \in \rl^3$, we have that $\D_x^{\cN^\prime} = \widehat{g} \circ \D_x^\cN \circ \widehat{g}^{- 1}$. Thus, $\widehat{g}$ restricts to an isomorphism of Hermitian vector bundles, $\widehat{g}|_{E^\cN} \defeq g_* : E^\cN \rightarrow E^{\cN^\prime}$, such that $g_* \( \nabla^\cN, \Phi^\cN \) = \( \nabla^{\cN^\prime}, \Phi^{\cN^\prime} \)$.
\end{remark}

\smallskip

\begin{remark}
   By \cref{eq:unique_framing}, the framing of $\bT$ is unique, up to gauge, and by \Cref{remark:gauging}, gauge equivalent framings determine canonically equivalent Sobolev spaces. Thus, from now on, we drop the framing from the notation of $L_1^2 \( \bT, \cF \) \defeq L_1^2 \( \bT \)$.
\end{remark}

\smallskip

\begin{remark}
   When $n = N$, and thus for all $a \in \{ 1, 2, \ldots, n \}$, $r_a = 1$, and $\cN$ is in temporal gauge, then the above construction recovers the one by Hurtubise and Murray in \cite{hurtubise_construction_1989}*{Section~4}.
\end{remark}

\smallskip

Note that in \Cref{definition:nahm_construction} we claimed, without proof, that $\( \D_x^\cN \)_{x \in \rl^3}$ is a smooth family of Fredholm operators of index $- N$ and trivial cokernel. Since this still needs verification, let us for all $x \in \rl^3$ and smooth $\Psi = \( \uppsi_a \)_{a = 1}^{n - 1} \in L_1^2 \( \bT \)$ define
\begin{equation}
   \widetilde{\D}_x^\cN \( \Psi \) \eqdef \( i \dot{\uppsi}_a + \sum\limits_{\alpha = 1}^3 \( \( i T_\alpha^{(a)} - x_\alpha \id_{\cx^{m_a}} \) \otimes \sigma_\alpha \) \uppsi_a \)_{a = 1}^{n - 1} \oplus \( \fB_a \( \lim\limits_{\epsilon \rightarrow 0^+}  \uppsi_{a - 1} \( \lambda_a - \epsilon \) \) \)_{a = 2}^{n - 1}. \label{eq:D_def}
\end{equation}
which is a smooth element of $L^2 \( \bT \)$. In order to prove \Cref{theorem:nahm}, namely that the pair $\( \nabla^\cN, \Phi^\cN \)$ is an $\SU (N)$-monopole that satisfies \cref{eq:mono1,eq:mono3,eq:mono2} with the desired asymptotics, we first need to build up the necessary functional analytic framework for the Dirac--Nahm operator \eqref{eq:DN_op}, which we do next in \Cref{theorem:fredholm}.

\medskip

\begin{theorem}
\label{theorem:fredholm}
   For all $x \in \rl^3$ the following statements hold:
   \begin{enumerate}

      \item The operator $\widetilde{\D}_x^\cN$ has a unique extension as a continuous linear map from $L_1^2 \( \bT \)$ to $L^2 \( \bT \)$, denoted by $\D_x^\cN$.

      \item The kernel of $\D_x^\cN$ is trivial.

      \item The operator $\D_x^\cN$ is Fredholm and $\index \( \D_x^\cN \) = - N$.

   \end{enumerate}

   Furthermore, the assignment $x \mapsto \D_x^\cN$ is analytic in the operator norm.
\end{theorem}

\begin{proof}
   In order to prove the first point, we show that the map in \cref{eq:D_def} has finite operator norm. Let $h^a$ be the Hermitian structure on $\( \lambda_a, \lambda_{a + 1} \) \times \cx^{m_a} \otimes \bH$. Let $\Psi = \( \uppsi_a \)_{a = 1}^{n - 1} \in L_1^2 \( \bT \)$ be smooth and for each $a \in \{ 2, 3, \ldots, n - 1 \}$, let
   \begin{equation}
      \upphi_a \eqdef \lim\limits_{\epsilon \rightarrow 0^+} \uppsi_{a - 1} \( \lambda_a - \epsilon \).
   \end{equation}
   Note that $\upphi_a \in Y_{a - 1}^+ \oplus Z_{a - 1}^+$ and that
   \begin{equation}
      \( C_a \otimes \id_\bH \) \( \upphi_a \) = \lim\limits_{\epsilon \rightarrow 0^+} \uppsi_a \( \lambda_a + \epsilon \) \in Y_a^- \oplus Z_a^- ,
   \end{equation}
   by \cref{cond:nahm1,cond:nahm2,cond:nahm3}. Now
   \begin{align}
      \| \widetilde{\D}_x^\cN \( \Psi \) \|_{L^2 \( \bT \)}^2  &= \sum\limits_{a = 1}^{n - 1} \| i \dot{\uppsi}_a + \sum\limits_{\alpha = 1}^3 \( \( i T_\alpha^{(a)} - x_\alpha \id_{\cx^{m_a}} \) \otimes \sigma_\alpha \) \uppsi_a \|_{L^2 \( \lambda_a, \lambda_{a + 1} \)}^2 + 
      \sum\limits_{a = 2}^{n - 1} \left| \fB_a \( \upphi_a \) \right|^2 \\
         &= \underbrace{\sum\limits_{a = 1}^{n - 1} \| \dot{\uppsi}_a \|_{L^2 \( \lambda_a, \lambda_{a + 1} \)}^2 + \sum\limits_{a = 2}^{n - 1} \left| \fB_a \( \upphi_a \) \right|^2}_{\cI_1} \\
         & \quad + \underbrace{\sum\limits_{a = 1}^{n - 1} \sum\limits_{\alpha = 1}^3 2 \: \Re \( \left\langle i \dot{\uppsi}_a \middle| \( \( i T_\alpha^{(a)} - x_\alpha \id_{\cx^{m_a}} \) \otimes \sigma_\alpha \) \uppsi_a \right\rangle_{L^2 \( \lambda_a, \lambda_{a + 1} \)} \)}_{\cI_2} \\
         & \quad + \underbrace{\sum\limits_{a = 1}^{n - 1} \| \sum\limits_{\alpha = 1}^3 \( \( i T_\alpha^{(a)} - x_\alpha \id_{\cx^{m_a}} \) \otimes \sigma_\alpha \) \uppsi_a \|_{L^2 \( \lambda_a, \lambda_{a + 1} \)}^2}_{\cI_3}.
   \end{align}
   Note that $\cI_1$ is already nonnegative.

   Using integration by parts and the skew-adjointness of $\sigma_\alpha$, the corresponding summand to each $a \in \{ 1, 2, \ldots, n - 1 \}$ and $\alpha \in \{ 1, 2, 3 \}$ in $\cI_2$ can be rewritten as
   \begin{align}
      \cI_{2, a, \alpha}   &\eqdef 2 \: \Re \( \left\langle i \dot{\uppsi}_a \middle| \( \( i T_\alpha^{(a)} - x_\alpha \id_{\cx^{m_a}} \) \otimes \sigma_\alpha \) \uppsi_a \right\rangle_{L^2 \( \lambda_a, \lambda_{a + 1} \)} \) \\
            &= \tfrac{1}{2} \cI_{2, a, \alpha} + \Re \( \left\langle i \dot{\uppsi}_a \middle| \( \( i T_\alpha^{(a)} - x_\alpha \id_{\cx^{m_a}} \) \otimes \sigma_\alpha \) \uppsi_a \right\rangle_{L^2 \( \lambda_a, \lambda_{a + 1} \)} \) \\
            &= \tfrac{1}{2} \cI_{2, a, \alpha} - \Re \( \left\langle \( \( i T_\alpha^{(a)} - x_\alpha \id_{\cx^{m_a}} \) \otimes \sigma_\alpha \) \uppsi_a \middle| i \dot{\uppsi}_a \right\rangle_{L^2 \( \lambda_a, \lambda_{a + 1} \)} \) \\
            & \quad - \Re \( \left\langle \uppsi_a \middle| \( \dot{T}_\alpha^{(a)} \otimes \sigma_\alpha \) \uppsi_a \right\rangle_{L^2 \( \lambda_a, \lambda_{a + 1} \)} \) \\
            & \quad + \( \lim\limits_{t \rightarrow \lambda_{a + 1}^-} - \lim\limits_{t \rightarrow \lambda_a^+} \) \Re \( h^a \( \uppsi_a (t), \( \( T_\alpha^{(a)} (t) + i x_\alpha \id_{\cx^{m_a}} \) \otimes \sigma_\alpha \) \uppsi_a (t) \) \) \\
            &= \tfrac{1}{2} \cI_{2, a, \alpha} - \tfrac{1}{2} \cI_{2, a, \alpha} - \Re \( \left\langle \uppsi_a \middle| \( \dot{T}_\alpha^{(a)} \otimes \sigma_\alpha \) \uppsi_a \right\rangle_{L^2 \( \lambda_a, \lambda_{a + 1} \)} \) \\
            & \quad + \lim\limits_{\epsilon \rightarrow 0^+} \Re \( h^a \( \Pi_{X_a^+} \uppsi_a, \Pi_{X_a^+} \( T_\alpha^{(a)} \otimes \sigma_\alpha \) \uppsi_a \)_{\lambda_{a + 1} - \epsilon} \) \\
            & \quad - \lim\limits_{\epsilon \rightarrow 0^+} \Re \( h^a \( \Pi_{X_a^-} \uppsi_a, \Pi_{X_a^-} \( T_\alpha^{(a)} \otimes \sigma_\alpha \) \uppsi_a \)_{\lambda_a + \epsilon} \) \\
            & \quad + \lim\limits_{\epsilon \rightarrow 0^+} \Re \( h^a \( \Pi_{Y_a^+} \uppsi_a, \Pi_{Y_a^+} \( T_\alpha^{(a)} \otimes \sigma_\alpha \) \uppsi_a \)_{\lambda_{a + 1} - \epsilon} \) \\
            & \quad - \lim\limits_{\epsilon \rightarrow 0^+} \Re \( h^a \( \Pi_{Y_a^-} \uppsi_a, \Pi_{Y_a^-} \( T_\alpha^{(a)} \otimes \sigma_\alpha \) \uppsi_a \)_{\lambda_a + \epsilon} \) \\
            & \quad + \lim\limits_{\epsilon \rightarrow 0^+} \Re \( h^a \( \Pi_{Z_a^+} \uppsi_a, \Pi_{Z_a^+} \( T_\alpha^{(a)} \otimes \sigma_\alpha \) \uppsi_a \)_{\lambda_{a + 1} - \epsilon} \) \\
            & \quad - \lim\limits_{\epsilon \rightarrow 0^+} \Re \( h^a \( \Pi_{Z_a^-} \uppsi_a, \Pi_{Z_a^-} \( T_\alpha^{(a)} \otimes \sigma_\alpha \) \uppsi_a \)_{\lambda_a + \epsilon} \).
   \end{align}
   Since $\Psi$ is smooth and satisfies \cref{cond:nahm1,cond:nahm2,cond:nahm3}, we get that the boundary terms corresponding to $X_a^\pm$ all vanish and using also \cref{eq:nahm_pole_+,eq:nahm_pole_-} we also get that the boundary terms corresponding to $Z_a^\pm$ all cancel after summing for all relevant $a, \alpha$. Using these observations and \cref{eq:B_a_def}, we get
   \begin{align}
      \cI_2 &= - \sum\limits_{a = 1}^{n - 1} \sum\limits_{\alpha = 1}^3 \Re \( \left\langle \uppsi_a \middle| \( \dot{T}_\alpha^{(a)} \otimes \sigma_\alpha \) \uppsi_a \right\rangle_{L^2 \( \lambda_a, \lambda_{a + 1} \)} \) \\
            & \quad + \sum\limits_{a = 1}^{n - 1} \sum\limits_{\alpha = 1}^3 \lim\limits_{\epsilon \rightarrow 0^+} \Re \( h^a \( \Pi_{Y_a^+} \uppsi_a, \Pi_{Y_a^+} \( T_\alpha^{(a)} \otimes \sigma_\alpha \) \uppsi_a \)_{\lambda_{a + 1} - \epsilon} \) \\
            & \quad - \sum\limits_{a = 1}^{n - 1} \sum\limits_{\alpha = 1}^3 \lim\limits_{\epsilon \rightarrow 0^+} \Re \( h^a \( \Pi_{Y_a^-} \uppsi_a, \Pi_{Y_a^-} \( T_\alpha^{(a)} \otimes \sigma_\alpha \) \uppsi_a \)_{\lambda_a + \epsilon} \) \\
            &= - \sum\limits_{a = 1}^{n - 1} \sum\limits_{\alpha = 1}^3 \Re \( \left\langle \uppsi_a \middle| \( \dot{T}_\alpha^{(a)} \otimes \sigma_\alpha \) \uppsi_a \right\rangle_{L^2 \( \lambda_a, \lambda_{a + 1} \)} \) - \sum\limits_{a = 2}^{n - 1} \Re \( h^a \( \upphi_a, \fC_a \( \upphi_a \) \) \) \\
            &= - \sum\limits_{a = 1}^{n - 1} \sum\limits_{\alpha = 1}^3 \Re \( \left\langle \uppsi_a \middle| \( \dot{T}_\alpha^{(a)} \otimes \sigma_\alpha \) \uppsi_a \right\rangle_{L^2 \( \lambda_a, \lambda_{a + 1} \)} \) - \sum\limits_{a = 2}^{n - 1} \left| \fB_a \( \upphi_a \) \right|^2.
   \end{align}

   Now for each $a \in \{ 1, 2, \ldots, n - 1 \}$, using $\sigma_\alpha \sigma_\beta = - \delta_{\alpha, \beta} + \sum_{\gamma = 1}^3 \epsilon_{\alpha, \beta, \gamma} \sigma_\gamma$ and \cref{eq:nahm}, the corresponding summand in the third term can be rewritten as
   \begin{align}
      \cI_{3, a}  &\eqdef \| \sum\limits_{\alpha = 1}^3 \( \( i T_\alpha^{(a)} - x_\alpha \id_{\cx^{m_a}} \) \otimes \sigma_\alpha \) \uppsi_a \|_{L^2 \( \lambda_a, \lambda_{a + 1} \)}^2 \\
                  &= - \sum\limits_{\alpha = 1}^3 \sum\limits_{\beta = 1}^3 \Re \( \left\langle \uppsi_a \middle| \( \( i T_\alpha^{(a)} - x_\alpha \id_{\cx^{m_a}} \) \otimes \sigma_\alpha \) \( \( i T_\beta^{(a)} - x_\beta \id_{\cx^{m_a}} \) \otimes \sigma_\beta \) \uppsi_a \right\rangle_{L^2 \( \lambda_a, \lambda_{a + 1} \)} \) \\
                  &= - \sum\limits_{\alpha = 1}^3 \sum\limits_{\beta = 1}^3 \Re \( \left\langle \uppsi_a \middle| \( \( i T_\alpha^{(a)} - x_\alpha \id_{\cx^{m_a}} \) \( i T_\beta^{(a)} - x_\beta \id_{\cx^{m_a}} \) \otimes \( \sigma_\alpha \sigma_\beta \) \) \uppsi_a \right\rangle_{L^2 \( \lambda_a, \lambda_{a + 1} \)} \) \\
                  &= - \sum\limits_{\alpha = 1}^3 \sum\limits_{\beta = 1}^3 \( - \delta_{\alpha, \beta} \) \Re \( \left\langle \uppsi_a \middle| \( \( i T_\alpha^{(a)} - x_\alpha \id_{\cx^{m_a}} \) \( i T_\beta^{(a)} - x_\beta \id_{\cx^{m_a}} \) \otimes \id_\bH \) \uppsi_a \right\rangle_{L^2 \( \lambda_a, \lambda_{a + 1} \)} \) \\
                  & \quad - \sum\limits_{\alpha = 1}^3 \sum\limits_{\beta = 1}^3 \sum\limits_{\gamma = 1}^3 \epsilon_{\alpha, \beta, \gamma} \Re \( \left\langle \uppsi_a \middle| \( \( i T_\alpha^{(a)} - x_\alpha \id_{\cx^{m_a}} \) \( i T_\beta^{(a)} - x_\beta \id_{\cx^{m_a}} \) \otimes \sigma_\gamma \) \uppsi_a \right\rangle_{L^2 \( \lambda_a, \lambda_{a + 1} \)} \) \\
                  &= \sum\limits_{\alpha = 1}^3 \left\| \( i T_\alpha^{(a)} - x_\alpha \id_{\cx^{m_a}} \) \uppsi_a \right\|_{L^2 \( \lambda_a, \lambda_{a + 1} \)}^2 \\
                  & \quad - \sum\limits_{\alpha = 1}^3 \sum\limits_{\beta = 1}^3 \sum\limits_{\gamma = 1}^3 \epsilon_{\alpha, \beta, \gamma} \Re \( \left\langle \uppsi_a \middle| \( \( i T_\alpha^{(a)} - x_\alpha \id_{\cx^{m_a}} \) \( i T_\beta^{(a)} - x_\beta \id_{\cx^{m_a}} \) \otimes \sigma_\gamma \) \uppsi_a \right\rangle_{L^2 \( \lambda_a, \lambda_{a + 1} \)} \) \\
                  &= \sum\limits_{\alpha = 1}^3 \left\| \( i T_\alpha^{(a)} - x_\alpha \id_{\cx^{m_a}} \) \uppsi_a \right\|_{L^2 \( \lambda_a, \lambda_{a + 1} \)}^2 \\
                  & \quad + \sum\limits_{\alpha = 1}^3 \sum\limits_{\beta = 1}^3 \sum\limits_{\gamma = 1}^3 \tfrac{1}{2} \epsilon_{\alpha, \beta, \gamma} \Re \( \left\langle \uppsi_a \middle| \( \left[ T_\alpha^{(a)}, T_\beta^{(a)} \right] \otimes \sigma_\gamma \) \uppsi_a \right\rangle_{L^2 \( \lambda_a, \lambda_{a + 1} \)} \) \\
                  & \quad + \sum\limits_{\alpha = 1}^3 \sum\limits_{\beta = 1}^3 \sum\limits_{\gamma = 1}^3 \epsilon_{\alpha, \beta, \gamma} \Re \( \left\langle \uppsi_a \middle| \( \( - i \( x_\alpha T_\beta^{(a)} + x_\beta T_\alpha^{(a)} \) + x_\alpha x_\beta \id_{\cx^{m_a}} \) \otimes \sigma_\gamma \) \uppsi_a \right\rangle_{L^2 \( \lambda_a, \lambda_{a + 1} \)} \).
   \end{align}
   Using the antisymmetry of $\epsilon_{\alpha, \beta, \gamma}$, the sum in the last line vanishes. Furthermore, using Nahm's \cref{eq:nahm}, we find
   \begin{equation}
      \cI_2 + \cI_3 = \sum\limits_{a = 1}^{n - 1} \sum\limits_{\alpha = 1}^3 \| \( \( i T_\alpha^{(a)} - x_\alpha \id_{\cx^{m_a}} \) \otimes \id_\bH \) \uppsi_a \|_{L^2 \( \lambda_a, \lambda_{a + 1} \)}^2 - \sum\limits_{a = 2}^{n - 1} \left| \fB_a \( \upphi_a \) \right|^2. \label{eq:half_of_weitzenbock}
   \end{equation}
   Since $\| \widetilde{\D}_x^\cN \( \Psi \) \|_{L^2 \( \bT \)}^2 = \cI_1 + \cI_2 + \cI_3$ and \cref{eq:half_of_weitzenbock} holds, we get that
   \begin{equation}
      \| \widetilde{\D}_x^\cN \( \Psi \) \|_{L^2 \( \bT \)}^2 = \sum\limits_{a = 1}^{n - 1} \| \dot{\uppsi}_a \|_{L^2 \( \lambda_a, \lambda_{a + 1} \)}^2 + \sum\limits_{a = 1}^{n - 1} \sum\limits_{\alpha = 1}^3 \| \( \( i T_\alpha^{(a)} - x_\alpha \id_{\cx^{m_a}} \) \otimes \id_\bH \) \uppsi_a \|_{L^2 \( \lambda_a, \lambda_{a + 1} \)}^2. \label{eq:weitzenbock}
   \end{equation}

   For the rest of the paper, let $r \eqdef |x|$ be the radial coordinate and let us recall Hardy's inequality: if $u \in L_1^2 \( \left[ 0, \epsilon \right] \)$, such that $u (0) = 0$, then
   \begin{equation}
      \int\limits_0^\epsilon \frac{|u (t)|^2}{t^2} \rd t \leqslant 4 \int\limits_0^\epsilon |\dot{u} (t)|^2 \rd t. \label[ineq]{ineq:hardy}
   \end{equation}
   Using \cref{ineq:hardy}, we get that there is a positive number $K^\cN$, such that
   \begin{equation}
      \( 0 \leqslant \) \sum\limits_{a = 1}^{n - 1} \sum\limits_{\alpha = 1}^3 \| \( \( i T_\alpha^{(a)} - x_\alpha \id_{\cx^{m_a}} \) \otimes \id_\bH \) \uppsi_a \|_{L^2 \( \lambda_a, \lambda_{a + 1} \)}^2 \leqslant \( K^\cN + r^2 \) \| \Psi \|_{L_1^2 \( \bT \)}^2. \label[ineq]{ineq:hardy_for_nahm}
   \end{equation}

   Using \cref{eq:H1_norm,ineq:hardy_for_nahm} and that $L_1^2$ embeds into $C^0$ in one dimension, this yields, after potentially increasing $K^\cN$, that
   \begin{equation}
      \| \widetilde{\D}_x^\cN \( \Psi \) \|_{L^2 \( \bT \)}^2 \leqslant \( K^\cN + r^2 \) \| \Psi \|_{L_1^2 \( \bT \)}^2. \label[ineq]{ineq:D_x_norm}
   \end{equation}
   Since the above formula holds on a dense subset of $L_1^2 \( \bT \)$, the claim about the unique and continuous extension, $\D_x^\cN$, is true. Moreover, we have $\| \D_x^\cN \| \leqslant \sqrt{K^\cN + r^2}$. The first claim is now proved.

   Moreover, \cref{eq:weitzenbock} implies that if $\D_x^\cN \( \Psi \) = 0$, then for all relevant $a , \alpha$, we get
   \begin{equation}
      \dot{\uppsi}_a = 0, \quad \( \( i T_\alpha^{(a)} - x_\alpha \id_{\cx^{m_a}} \) \otimes \id_\bH \) \uppsi_a = 0.
   \end{equation}
   Thus, for an irreducible Nahm data, $\uppsi_a$ must vanish, and hence $\ker \( \D_x^\cN \)$ is trivial. The second claim is now proved.

   Furthermore, \cref{eq:weitzenbock} combined with \cite{buhler_functional_2018}*{Lemma 4.3.9} and the compactness of $L_1^2 \hookrightarrow L^2$ in dimension one, yields that the image of $\D_x^\cN$ is closed in $L^2 \( \bT \)$. Thus, to prove the third claim, it is enough to prove that $\dim \( \coker \( \D_x^\cN \) \) = N$. Since $L^2 \( \bT \)$ and $L_1^2 \( \bT \)$ are Hilbert spaces and $\image \( \D_x^\cN \)$ is closed in $L^2 \( \bT \)$, we can canonically identify the cokernel of $\D_x^\cN$ with the kernel of $\( \D_x^\cN \)^*$. First, let us compute the action of $\( \D_x^\cN \)^*$ on a $\Psi = \( \( \uppsi_a \)_{a = 1}^{n - 1}, \( \upphi_a \)_{a = 2}^{n - 1} \) \in L^2 \( \bT \) \cap L_1^2 \( \bT \)$. 
   For convenience, for all $b \in \left\{ 2, 3, \ldots, n - 1 \right\}$, let
   \begin{equation}
      \tr_\bT \( \Psi \)_b \eqdef \lim\limits_{\epsilon \rightarrow 0^+} \Pi_{Y_b^+ \oplus Z_b^+} \( i \( C_b^* \otimes \id_\bH \) \( \uppsi_b \( \lambda_b + \epsilon \) \) - i \uppsi_{b - 1} \( \lambda_b - \epsilon \) \) + \fB_b \( \upphi_b \) \in Y_b^+ \oplus Z_b^+,
   \end{equation}
   and $\tr_\bT \( \Psi \) \eqdef \( \tr_\bT \( \Psi \)_b \)_{b = 2}^{n - 1}$. Then for all $\widetilde{\Psi} = \( \widetilde{\uppsi}_a \)_{a = 1}^{n - 1} \in L_1^2 \( \bT \)$ we have, using integration by parts and \cref{cond:nahm1,cond:nahm2,cond:nahm3}, that
   \begin{align}
      \left\langle \widetilde{\Psi} \middle| \( \D_x^\cN \)^* \( \Psi \) \right\rangle_{L_1^2 \( \bT \)} &= \left\langle \D_x^\cN \( \widetilde{\Psi} \) \middle| \Psi \right\rangle_{L^2 \( \bT \)} \\
         &= \sum\limits_{a = 1}^{n - 1} \int\limits_{\lambda_a}^{\lambda_{a + 1}} h^a \( i \dot{\widetilde{\uppsi}}_a (t) + \sum\limits_{\alpha = 1}^3 \( \( i T_\alpha^{(a)} (t) - x_\alpha \id_{\cx^{m_a}} \) \otimes \sigma_\alpha \) \widetilde{\uppsi}_a (t), \uppsi_a (t) \) \rd t \\
         & \quad + \sum\limits_{b = 2}^{n - 1} h^{a - 1} \( \fB_b \( \lim\limits_{\epsilon \rightarrow 0^+} \widetilde{\uppsi}_{b - 1} \( \lambda_a - \epsilon \) \), \upphi_b \) \\
         &= \sum\limits_{a = 1}^{n - 1} \int\limits_{\lambda_a}^{\lambda_{a + 1}} h^a \( \widetilde{\uppsi}_a (t), i \dot{\uppsi}_a (t) - \sum\limits_{\alpha = 1}^3 \( \( i T_\alpha^{(a)} (t) - x_\alpha \id_{\cx^{m_a}} \) \otimes \sigma_\alpha \) \uppsi_a (t) \) \rd t \\
         & \quad + \sum\limits_{b = 2}^{n - 1} \lim\limits_{\epsilon \rightarrow 0^+} h^{b - 1} \( \widetilde{\uppsi}_{b - 1} \( \lambda_b - \epsilon \), \tr_\bT \( \Psi \)_b \).
   \end{align}
   Now if $\Psi = \( \( \uppsi_a \)_{a = 1}^{n - 1}, \( \upphi_b \)_{b = 2}^{n - 1} \) \in \ker \( \( \D_x^\cN \)^* \)$, then it has to satisfy
   \begin{subequations}
   \begin{equation}
      0 = i \dot{\uppsi}_a - \sum\limits_{\alpha = 1}^3 \( \( i T_\alpha^{(a)} - x_\alpha \id_{\cx^{m_a}} \) \otimes \sigma_\alpha \) \uppsi_a, \label{eq:strong_ODE}
   \end{equation}
   on $\( \lambda_a, \lambda_{a + 1} \)$, and in particular the weak derivatives $\dot{\uppsi}_a$ exist. In fact, by elliptic bootstrap, $\uppsi_a$ is smooth. Furthermore, $\Psi$ must also satisfy $\tr_\bT \( \Psi \) = 0$, or equivalently, for all $b \in \{ 2, \ldots, n - 1 \}$, we have
   \begin{equation}
      \lim\limits_{\epsilon \rightarrow 0^+} \( \( C_b^* \otimes \id_\bH \) \( \uppsi_b \( \lambda_b + \epsilon \) \) - \Pi_{Y_b^+ \oplus Z_b^+} \( \uppsi_{b - 1} \( \lambda_b - \epsilon \) \) \) = i \fB_b \( \upphi_b \) \in Y_b^+. \label{eq:pre_BVP_a}
   \end{equation}
   \end{subequations} 

   By the Picard--Lindel\"of Theorem, we have a $2 m_a$-dimensional space of solutions to \cref{eq:strong_ODE}. The theory of ordinary differential equations with regular singularities, as used in \cites{hitchinMonopoles,hurtubise_construction_1989,Cherkis-Larrain-Stern-2}, tells us that this $2 m_a$-dimensional space contains various subspaces, as follows: there is a $\( 2 m_a - 2 k_a^- \)$-dimensional space of solutions that have right limits in $Y_a^- \oplus Z_a^-$ at $\lambda_a$ from the right. Furthermore, for each summand $R_{k_{a, b}^-}$ in \cref{eq:reduction} there is a $\( k_{a, b}^- + 1 \)$-dimensional space of solutions that are $O \( \( t - \lambda_a \)^{\( k_{a, b}^- - 1 \) / 2} \)$, thus are in $L^2$ and satisfy \cref{cond:nahm1}, and another $\( k_{a, b}^- - 1 \)$-dimensional space of solutions that are $O \( \( t - \lambda_a \)^{- \( k_{a, b}^- + 1 \) / 2} \)$, and thus not in $L^2$, near $\lambda_a$ from the right.

   Similarly, there is a $\( 2 m_a - 2 k_{a + 1}^+ \)$-dimensional space of solutions that have left limits in $Y_a^+ \oplus Z_a^+$ at $\lambda_{a + 1}$ from the left. Furthermore, for each summand $R_{k_{a + 1, b}^+}$ in \cref{eq:reduction} there is a $\( k_{a + 1, b}^+ + 1 \)$-dimensional space of solutions that are $O \( \( \lambda_{a + 1} - t \)^{\( k_{a+1, b}^+ - 1 \) / 2} \)$, thus are in $L^2$ and satisfy \cref{cond:nahm2}, and another $\( k_{a + 1, b}^+ - 1 \)$-dimensional space of solutions that are $O \( \( \lambda_{a + 1} - t \)^{- \( k_{a+1, b}^+ + 1 \) / 2} \)$, and thus not in $L^2$, near $\lambda_{a + 1}$ from the left.

   In summary, there is a $\( 2 m_a - k_a^- + r_a^- \)$-dimensional space of solutions that are in $L^2$ near $\lambda_a$ from the left and all of these have limits in $Y_a^- \oplus Z_a^-$ and there is a $\( 2 m_a - k_{a + 1}^+ + r_{a + 1}^+ \)$-dimensional space of solutions that are in $L^2$ near $\lambda_{a + 1}$ from the right and all of these have limits in $Y_{a + 1}^+ \oplus Z_{a + 1}^+$. For all relevant $a$, let $X_{a, 1}^\pm \leqslant X_a^\pm$ be the the subspace corresponding to rank-$1$ representations in \cref{eq:reduction}. Using \cref{eq:pre_BVP_a}, we have that solution that are in $L^2$ satisfy, for all $a \in \{ 2, \ldots, n - 1 \}$, that
   \begin{subequations}
   \begin{align}
      \lim\limits_{\epsilon \rightarrow 0^+} \uppsi_1 \( \lambda_1 + \epsilon \)                                                                                                                &\in X_{1, 1}^-, \label{eq:BVP_1} \\
      \lim\limits_{\epsilon \rightarrow 0^+} \( \( C_a^* \otimes \id_\bH \) \( \uppsi_a \( \lambda_a + \epsilon \) \) - \uppsi_{a - 1} \( \lambda_a - \epsilon \) \) - i \fB_a \( \upphi_a \)   &\in X_{a, 1}^+, \label{eq:BVP_a_+} \\
      \lim\limits_{\epsilon \rightarrow 0^+} \( \uppsi_a \( \lambda_a + \epsilon \) - \( C_a \otimes \id_\bH \) \( \uppsi_{a - 1} \( \lambda_a - \epsilon \) - i \fB_a \( \upphi_a \) \) \)     &\in X_{a, 1}^-, \label{eq:BVP_a_-} \\
      \lim\limits_{\epsilon \rightarrow 0^+} \uppsi_{n - 1} \( \lambda_n - \epsilon \)                                                                                                          &\in X_{n, 1}^+, \label{eq:BVP_n}
   \end{align}
   \end{subequations}
   Informally, \cref{eq:BVP_1,eq:BVP_a_+,eq:BVP_a_-,eq:BVP_n} can be summarized as follows:
   \begin{itemize}
      \item Components in $X_a^\pm \cap \( X_{a, 1}^\pm \)^\perp$ must converge to zero at the endpoints.
      \item Components in $X_{a, 1}^\pm \oplus Y_a^\pm$ must converge, but there are no further restrictions, and $\upphi_a$ is given by these limits.
      \item Components in $Z_a^\pm$ must converge and are $C_a$-continuous, in the sense of \cref{eq:BVP_a_+,eq:BVP_a_-}.
   \end{itemize}
   For each $a \in \{ 1, 2, \ldots, n - 1 \}$, let $\cV_a$ be the $2 m_a$-dimensional space of solutions to \cref{eq:strong_ODE}. Let $\cU_1 \subseteq \cV_1$ be the $\( k_1^- + r_1^- \)$-dimensional subspace of solutions that are $L^2$ near $\lambda_1$. For each $a \in \{ 2, 3, \ldots, n - 1 \}$, let $\cU_a \subseteq \cV_{a - 1} \times \cV_a$ be the subspace of pairs of solutions that are both $L^2$ near $\lambda_a$ and satisfy \cref{eq:BVP_a_+,eq:BVP_a_-}, and let $\cU_n \subseteq \cV_{n - 1}$ be the $\( k_n^+ + r_n^+ \)$-dimensional subspace of solutions that are $L^2$ near $\lambda_n$. By the above discussion, for all $a \in \{ 2, 3, \ldots, n - 2 \}$, we have
   \begin{equation}
      \dim \( \cU_a \) = \underbrace{2 m_{a - 1} - 2 k_a^+}_{\textnormal{continuous at $\lambda_a$}} + \underbrace{k_a^+ + r_a^+}_{\substack{\textnormal{vanishing from} \\ \textnormal{the left at $\lambda_a$}}} + \underbrace{k_a^- + r_a^-}_{\substack{\textnormal{vanishing from} \\ \textnormal{the right at $\lambda_a$}}} + \underbrace{r_a^0}_{\textnormal{jumping at $\lambda_a$}} = m_{a - 1} + m_a + r_a.
   \end{equation}

   Finally, let us define the map
   \begin{equation}
      \cD \colon \bigoplus_{a = 1}^n \cU_a \rightarrow \bigoplus_{a = 1}^{n - 1} \cV_a; \ \( \uppsi_1^+, \( \uppsi_1^-, \uppsi_2^+ \), \ldots, \( \uppsi_{n - 2}^-, \uppsi_{n - 1}^+ \), \uppsi_{n - 1}^- \) \mapsto \( \uppsi_a^+ - \uppsi_a^- \)_{a = 1}^{n - 1}.
   \end{equation}
   Since $\cD$ is a map between finite dimensional vector spaces, its index is (using that $r_1^- = r_1$ and $r_n^+ = r_n$)
   \begin{align}
      \index \( \cD \)  &= \dim \( \bigoplus_{a = 1}^n \cU_a \) - \dim \( \bigoplus_{a = 1}^{n - 1} \cV_a \) \\
                        &= \sum\limits_{a = 1}^n \dim \( \cU_a \) - \sum\limits_{a = 1}^{n - 1} \dim \( \cV_a \) \\
                        &= \( k_1^- + r_1^- \) + \sum\limits_{a = 2}^{n - 1} \( m_{a - 1} + m_a + r_a \) + \( k_n^+ + r_n^+ \) - \sum\limits_{a = 1}^{n - 1} 2 m_a \\
                        &= \sum\limits_{a = 1}^n r_a \\
                        &= N.
   \end{align}
   Furthermore, $\ker \( \cD \) \cong \ker \( \( \D_x^\cN \)^* \)$ via the map
   \begin{equation}
      \ker \( \cD \) \rightarrow \ker \( \( \D_x^\cN \)^* \); \ \( \uppsi_1^+, \( \uppsi_1^-, \uppsi_2^+ \), \ldots, \( \uppsi_{n - 2}^-, \uppsi_{n - 1}^+ \), \uppsi_{n - 1}^- \) \mapsto \( \uppsi_a^+ \)_{a = 1}^{n - 1}.
   \end{equation}
   Hence, it is enough to show that $\coker \( \cD \)$ is trivial. For each $a \in \{ 1, 2, \ldots, n - 1 \}$, let $\widetilde{\cV}_a$ be the $2 m_a$-dimensional space of solutions to
   \begin{equation}
      i \dot{\eta}_a + \sum\limits_{\alpha = 1}^3 \( \( i T_\alpha^{(a)} - x_\alpha \id_{\cx^{m_a}} \) \otimes \sigma_\alpha \) \eta_a = 0,
   \end{equation}
   on $\( \lambda_a, \lambda_{a + 1} \)$. If $\eta_a \in \widetilde{\cV}_a$ and $\uppsi_a \in \cV_a$, then $h^a \( \eta_a, \uppsi_a \)$ is constant. In fact, using the Picard--Lindel\"of Theorem again, we get that $\widetilde{\cV}_a \cong \cV_a^*$ via $\eta_a \mapsto \( \uppsi_a \mapsto h^a \( \eta_a, \uppsi_a \)|_{t = t_0} \)$, for arbitrary $t_0 \in \( \lambda_a, \lambda_{a + 1} \)$. Now we have an isomorphism
   \begin{equation}
      \bigoplus_{a = 1}^{n - 1} \widetilde{\cV}_a \cong \( \bigoplus_{a = 1}^{n - 1} \cV_a \)^*. \label{eq:iso_1}
   \end{equation}
   Furthermore, we have that
   \begin{equation}
      \( \coker \( \cD \) \)^* \cong \left\{ \ f \in \( \bigoplus_{a = 1}^{n - 1} \cV_a \)^* \ \middle| \ f|_{\image \( \cD \)} = 0 \ \right\} \label{eq:iso_2}
   \end{equation}
   through the isomorphism  $F \mapsto F \circ \pi$ defined using the  canonical projection
   \begin{equation}
      \pi \colon \bigoplus_{a = 1}^{n - 1} \cV_a \rightarrow \( \bigoplus_{a = 1}^{n - 1} \cV_a \) \big/ \image \( \cD \).
   \end{equation}
   The  argument used in \cite{hurtubise_construction_1989}*{middle of page 79} shows that under the isomorphisms given in \cref{eq:iso_1,eq:iso_2}, $\( \coker \( \cD \) \)^*$ is identified with $\ker \( \D_x^\cN \)$, which we have already shown to be trivial. The second and third claims are now proved.

   Finally, for all $x, y \in \rl^3$, the operator $\D_x^\cN - \D_y^\cN$ is just the left-multiplication by the imaginary quaternion $\sum_{\alpha = 1}^3 \( x_\alpha - y_\alpha \) \sigma_\alpha$. In particular, $\| \D_x^\cN - \D_y^\cN \| = |x - y|$, and thus the assignment $x \mapsto \D_x^\cN$ is continuous and linear, and thus analytic. The proof is now complete.
\end{proof}

\smallskip

Next we construct (possibly well-known) solutions to a model equation for \cref{eq:strong_ODE}. Similar analysis was done in \cite{Cherkis-Larrain-Stern-2}*{Section~5.2}. 

Regard $\bH \cong \cx^2$ as the defining, irreducible representation of $\su (2)$. Let $\vee$ be the symmetrized tensor product on the (complex) tensor algebra of $\bH$. For each $k \in \Z_+$, we regard $\cx^k$ as the $k$-dimensional irreducible representation of $\su (2)$ in the following way:
\begin{equation}
   \cx^k \cong \left\{ \begin{array}{ll} \cx, & \mbox{if } k = 1, \\ \bH, & \mbox{if } k = 2, \\ \bigvee^k \bH = \spn_\cx \( \left\{ \: w_1 \vee w_2 \vee \cdots \vee w_{k - 1} \: \middle | \: w_1, w_2, \ldots, w_{k - 1} \in \bH \: \right\} \), & \mbox{if } k \geqslant 3, \end{array} \right.
\end{equation}
and for all $\sigma \in \su (2)$, we $\varrho_k \( \sigma \)$ act on $\cx^k$ via the extension of $\sigma$ as a derivation. Now let
\begin{equation}
   \widehat{C}_k \eqdef \sum\limits_{\alpha = 1}^3 \varrho_k \( \sigma_\alpha \) \otimes \sigma_\alpha \quad \& \quad \widehat{x} \eqdef \sum\limits_{\alpha = 1}^3 x_\alpha \sigma_\alpha.
\end{equation}

The model equation for \cref{eq:strong_ODE} is given as follows: for each $k \in \Z_+$ and $x \in \rl^3$, let us consider the first order differential equation for smooth $\psi \colon I \rightarrow \cx^k \otimes \bH$ of the form
\begin{equation}
   \dot{\uppsi} + \tfrac{1}{2 t} \widehat{C}_k \( \uppsi \) - i \id_{\cx^k} \otimes \widehat{x} \( \uppsi \) = 0, \label{eq:model_eq}
\end{equation}
where $I \subseteq \rl$ is an open interval.

\medskip

\begin{lemma}
   \label{lemma:simultaneous_EVEs}
   For each $k \in \Z_+$ and $x \in \rl^3 - \left\{ \textnormal{o} \right\}$, there exists vectors, $u \( k, x \) \in \cx^k \otimes \bH$, such that the following hold:
   \begin{subequations}
   \begin{align}
      \widehat{C}_k \( u \( k, x \) \)  &= \( 1 - k \) u \( k, x \), \label[cond]{cond:u_cond_1} \\
      \widehat{x} \( u \( k, x \) \)    &= ir u \( k, x \), \label[cond]{cond:u_cond_2} \\
      \left| u \( k, x \) \right|       &= 1. \label[cond]{cond:u_cond_3}
   \end{align}
   \end{subequations}
   Moreover $u \( k, x \)$ is unique up to a $\rU (1)$ factor. In fact, the assignment
   \begin{equation}
      \rl^3 - \left\{ \textnormal{o} \right\} \ni x \mapsto \spn_\cx \( \left\{ u \( k, x \) \right\} \),
   \end{equation}
   defines a Hermitian line bundle of Chern number $- k$.
\end{lemma}

\begin{proof}
   For all $x \in \rl^3 - \left\{ \textnormal{o} \right\}$ with $r = |x|$, let $\cL_x \subset \bH$ be the $i r$-eigenline of $\widehat{x}$. Then the Hermitian line bundle $\( \cL_x \)_{x \in \rl^3 - \left\{ \textnormal{o} \right\}}$ is the anti-Hopf fibration, thus has Chern number $- 1$.

   Let $w (x) \in \cL_x$ be a unit length vector. Let
   \begin{equation}
      u \( k, x \) \eqdef \( \vee^{k - 1} w (x) \) \otimes w (x) \in \cx^k \otimes \bH. \label{eq:u_kx}
   \end{equation}
   Here $\vee^0 w (x) = 1 \in \cx$ and $\vee^1 w (x) = w (x) \in \bH$. Now for all $k \in \Z_+$, the vector $u \( k, x \)$ trivially satisfies \cref{cond:u_cond_2,cond:u_cond_3}, so the only thing left to show is \cref{cond:u_cond_1}:
   \begin{align}
      \widehat{C}_k \( u \( k, x \) \)  &= \( \sum\limits_{\alpha = 1}^3 \varrho_k \( \sigma_\alpha \) \otimes \sigma_\alpha \) \( \( \vee^{k - 1} w (x) \) \otimes w (x) \) \\
         &= \sum\limits_{\alpha = 1}^3 \varrho_k \( \sigma_\alpha \) \( \vee^{k - 1} w (x) \) \otimes \sigma_\alpha \( w (x) \) \\
         &= \( k - 1 \) \( \vee^{k - 2} w (x) \) \vee \( \sum\limits_{\alpha = 1}^3 \sigma_\alpha \( w (x) \) \otimes \sigma_\alpha \( w (x) \) \) \\
         &= \( k - 1 \) \( \vee^{k - 2} w (x) \) \vee \( - \sum\limits_{\alpha = 1}^3 w (x) \otimes w (x) \) \\
         &= \( 1 - k \) u \( k, x \).
   \end{align}
   The uniqueness of $\spn_\cx \( \left\{ u \( k, x \) \right\} \)$ follows from the fact that \cref{cond:u_cond_1} implies that $u \( k, x \)$ ought to be in the $\( k + 1 \)$-dimensional irreducible component of $\cx^k \otimes \bH \cong \cx^{k + 1} \oplus \cx^{k - 1}$, while \cref{cond:u_cond_2} implies that $u \( k, x \)$ ought to be an elementary tensor. It is easy to verify that these two requirements can only be satisfied by a vector of the form \eqref{eq:u_kx}. Thus, $\( \spn_\cx \( \left\{ u \( k, x \) \right\} \) \)_{x \in \rl^3 - \left\{ \textnormal{o} \right\}}$ is a Hermitian line bundle that equals to $\cL^{\otimes k}$, which proves the claim about the Chern numbers. 
\end{proof}

\smallskip

\begin{remark}\label{remark:simultaneous_EVEsnegativek}
   Suppose now $k \in \Z_-$. The anti-unitary isomorphism, $J$, in \cref{eq:quaternionic_str}, extends to $\cx^{|k|}\otimes \bH$ and satisfies $J \circ \widehat{x} = \widehat{x} \circ J$. Thus, if we set $u \( k, x \) \eqdef J \( u \( |k|, x \) \)$, then $u \( k, x \)$ is the unique (up to a $\rU (1)$ factor) vector that satisfies
   \begin{subequations}
   \begin{align}
      \widehat{C}_{|k|} \( u \( k, x \) \)   &= \( 1 - |k| \) u \( k, x \), \label[cond2]{cond2:u_cond_1} \\
      \widehat{x} \( u \( k, x \) \)         &= - ir u \( k, x \), \label[cond2]{cond2:u_cond_2} \\
      \left| u \( k, x \) \right|            &= 1, \label[cond2]{cond2:u_cond_3}
   \end{align}
   \end{subequations}
   and that the assignment
   \begin{equation}
      \rl^3 - \left\{ \textnormal{o} \right\} \ni x \mapsto \spn_\cx \( \left\{ u \( k, x \) \right\} \),
   \end{equation}
   defines a Hermitian line bundle of Chern number $k$.
\end{remark}

Now we can prove the key lemma for constructing approximate solutions.

\begin{lemma}
   \label{lemma:model_solution_k_not_zero}
   For $k \in \Z_+$, let
   \begin{equation}
      \uppsi_{k, x} \colon \rl_+ \rightarrow \cx^k \otimes \bH; \ t \mapsto \tfrac{\( 2 r \)^{\frac{k}{2}}}{\sqrt{\( k - 1 \)!}} t^{\frac{k - 1}{2}} e^{- r t} u \( k, x \).
   \end{equation}
   Then $\uppsi_{k, x}$ solves the model \cref{eq:model_eq} with $I = \rl_+$, satisfies $\| \uppsi_{k, x} \|_{L^2 \( \rl_+ \)} = 1$, and
   \begin{equation}
      \int\limits_0^\infty t |\uppsi_{k, x} (t)|^2 \rd t = \tfrac{k}{2 r}. \label{eq:model_Higgs_+}
   \end{equation}
   For $k \in \Z_-$, let
   \begin{equation}
      \uppsi_{k, x} \colon \rl_- \rightarrow \cx^k \otimes \bH; \ t \mapsto J \( \uppsi_{|k|, x} (- t) \).
   \end{equation}
   Then $\uppsi_{k, x}$ solves the model \cref{eq:model_eq} with $I = \rl_-$, satisfies $\| \uppsi_{k, x} \|_{L^2 \( \rl_- \)} = 1$, and
   \begin{equation}
      \int\limits_{- \infty}^0 t |\uppsi_{k, x} (t)|^2 \rd t = \tfrac{k}{2 r}. \label{eq:model_Higgs_-}
   \end{equation}
\end{lemma}

\begin{proof}
	The fact that $\uppsi_{k, x}$ solves the model \cref{eq:model_eq} is immediate. If $k \in \Z_+$, then, since $\left| u \( k, x \) \right| = 1$, we have
   \begin{align}
      \int\limits_0^\infty t |\uppsi_{k, x} (t)|^2 \rd t &= \int\limits_0^\infty t \tfrac{\( 2 r \)^k}{\( k - 1 \)!} t^{k - 1} e^{- 2 r t} \rd t \\
         &= \tfrac{k}{2 r} \underbrace{\tfrac{1}{k!} \int\limits_0^\infty \( 2 r t \)^k e^{- \( 2 r t \)} \rd \( 2 r t \)}_{= 1} \\
         &= \tfrac{k}{2 r},
   \end{align}
   yielding \cref{eq:model_Higgs_+}. The case of negative $k$ follows analogously.
\end{proof}

\smallskip
\begin{remark}
   Note that $\uppsi_{\pm 1, x}$ does not vanish in the $t \rightarrow 0^\mp$ limit, as opposed to the $|k| > 1$ cases.
\end{remark}

\begin{remark}
   \label{remark:Dirac_mono}
   Fix $k \in \Z - \{ 0 \}$. If $k > 0$, then let
   \begin{equation}
      \rl^3 - \left\{ \textnormal{o} \right\} \ni x \mapsto \spn_\cx \( \left\{ \uppsi_{k, x} \right\} \) \defeq \cL_x^{(k)} \subset L^2 \( \rl_+, \cx^k \otimes \bH \),
   \end{equation}
   and if $k < 0$, then let
   \begin{equation}
      \rl^3 - \left\{ \textnormal{o} \right\} \ni x \mapsto \spn_\cx \( \left\{ \uppsi_{k, x} \right\} \) \defeq \cL_x^{(k)} \subset L^2 \( \rl_-, \cx^{-k} \otimes \bH \).
   \end{equation}
   Now for all $k \in \Z - \{ 0 \}$, $\cL^{(k)}$ defines a Hermitian line bundle of Chern number $- k$. Let $\nabla^k$ be the induced (Berry) connection and $\varphi^k \in \Gamma \( \End \( \cL^{(k)} \) \)$ be given by $\varphi^k \( \psi \)_x \eqdef \Pi_{\cL_x^{(k)}} \( M \psi \)$, where $M$ is the operator of multiplication by $t$. Then $\( \nabla^k, \varphi^k \)$ is the \emph{Dirac monopole} of charge $k$, centered at the origin.
\end{remark}

\medskip

We end this section with a technical lemma about the inverse of the operator $\( \D_x^\cN \)^* \D_x^\cN$. The proof of \cref{eq:weitzenbock} can easily be generalized to prove that for a smooth element $\Psi = \( \uppsi_a \)_{a = 1}^{n - 1} \in L_1^2 \( \bT \)$ we have
\begin{equation}
   H_x^\cN \( \Psi \) \eqdef \( \D_x^\cN \)^* \D_x^\cN \( \Psi \) = \( - \ddot{\uppsi}_a + \sum\limits_{\alpha = 1}^3 \( \( i T_\alpha^{(a)} - x_\alpha \id_{\cx^{m_a}} \)^2 \otimes \id_\bH \) \uppsi_a \)_{a = 1}^{n - 1}. \label{eq:second_order_ODE}
\end{equation}
Note that \cref{ineq:D_x_norm} implies $H_x^\cN$ has a continuous inverse (when regarded as a map to its topological dual). We denote by $G_x^\cN$ the restriction of this inverse to $L^2 \( \bT \)$, acting as zero on $\bigoplus\limits_{b = 1}^{n - 2} Y_b^+$.

\begin{lemma}
   \label{lemma:Greens_bound}
   The operator norm of $G_x^\cN$ depends on $x \in \rl^3$ analytically and satisfies
   \begin{equation}
      \| G_x^\cN \| = O \( r^{- 2} \). \label{eq:Greens_bound}
   \end{equation}
\end{lemma}

\begin{proof}
   Since $x \mapsto H_x^\cN$ is analytic by \Cref{theorem:fredholm}, the first claim is immediate.

   Let $\Psi \in L_1^2 \( \bT \)$. First we show that there is a positive number $C$, only dependent on $\cN$, such that
   \begin{equation}
      \| \D_x^\cN \( \Psi \) \|_{L^2 \( \bT \)}^2 \geqslant C r^2 \| \Psi \|_{L^2 \( \bT \)}^2. \label[ineq]{ineq:D_Psi_lower_bound_1}
   \end{equation}
   Let us write $\Psi = \( \uppsi_a \)_{a = 1}^{n - 1}$ and use \cref{eq:weitzenbock} and the Peter--Paul inequality (with $\delta > 0$) to get
   \begin{align}
      \| \D_x^\cN \( \Psi \) \|_{L^2 \( \bT \)}^2  &= \sum\limits_{a = 1}^{n - 1} \| \dot{\uppsi}_a \|_{L^2 \( \lambda_a, \lambda_{a + 1} \)}^2 + \sum\limits_{a = 1}^{n - 1} \sum\limits_{\alpha = 1}^3 \| \( \( i T_\alpha^{(a)} - x_\alpha \id_{\cx^{m_a}} \) \otimes \id_\bH \) \uppsi_a \|_{L^2 \( \lambda_a, \lambda_{a + 1} \)}^2 \\
         &= \sum\limits_{a = 1}^{n - 1} \| \dot{\uppsi}_a \|_{L^2 \( \lambda_a, \lambda_{a + 1} \)}^2 + r^2 \| \Psi \|_{L^2 \( \bT \)} + \overbrace{\sum\limits_{a = 1}^{n - 1} \sum\limits_{\alpha = 1}^3 \| \( i T_\alpha^{(a)} \otimes \id_\bH \) \uppsi_a \|_{L^2 \( \lambda_a, \lambda_{a + 1} \)}^2}^{I \eqdef} \\
         & \quad - 2 \sum\limits_{\alpha = 1}^3 \sum\limits_{a = 1}^{n - 1} \Re \( \langle \( i T_\alpha^{(a)} \otimes \id_\bH \) \uppsi_a | x_\alpha \uppsi_a \rangle_{L^2 \( \lambda_a, \lambda_{a + 1} \)} \) \\
         &\geqslant \sum\limits_{a = 1}^{n - 1} \| \dot{\uppsi}_a \|_{L^2 \( \lambda_a, \lambda_{a + 1} \)}^2 + r^2 \| \Psi \|_{L^2 \( \bT \)}^2 + I - 2 \( \tfrac{r^2 \delta}{2} \| \Psi \|_{L^2 \( \bT \)}^2 + \tfrac{1}{2 \delta} I \) \\
         &= \sum\limits_{a = 1}^{n - 1} \| \dot{\uppsi}_a \|_{L^2 \( \lambda_a, \lambda_{a + 1} \)}^2 + \( 1 - \delta \) r^2 \| \Psi \|_{L^2 \( \bT \)}^2 + \( 1 - \tfrac{1}{\delta} \) I. \label{ineq:D_Psi_lower_bound_2}
   \end{align}
   Using Hardy's \cref{ineq:hardy}, we also get that there is a positive number, $K$, only dependent on $\cN$, such that $\sum_{a = 1}^{n - 1} \| \dot{\uppsi}_a \|_{L^2 \( \lambda_a, \lambda_{a + 1} \)}^2 \geqslant K I$, and thus \cref{ineq:D_Psi_lower_bound_2} becomes
   \begin{align}
      \| \D_x^\cN \( \Psi \) \|_{L^2 \( \bT \)}^2 \geqslant \( 1 - \delta \) r^2 \| \Psi \|_{L^2 \( \bT \)}^2 + \( K + 1 - \tfrac{1}{\delta} \) I.
   \end{align}
   Now choosing $\delta \eqdef \tfrac{1}{1 + K}$ and $C \eqdef \tfrac{K}{1 + K}$ proves \cref{ineq:D_Psi_lower_bound_1}. Let $\Upsilon = H_x \( \Psi \) \Leftrightarrow \Psi = G_x^\cN \( \Upsilon \)$ and assume that $\| \Upsilon \|_{L^2 \( \bT \)} = 1$. Since $G_x$ is a positive, self-adjoint operator on $L^2 \( \bT \)$, we have that
   \begin{equation}
      \| G_x^\cN \| \geqslant \langle G_x^\cN \( \Upsilon \) | \Upsilon \rangle = \langle \Psi | H_x \( \Psi \) \rangle = \| \Psi \|_{L^2 \( \bT \)}^2 \geqslant C r^2 \| \Psi \|_{L^2 \( \bT \)}^2 = C r^2 \| G_x^\cN \( \Upsilon \) \|_{L^2 \( \bT \)}^2.
   \end{equation}
   By choosing a sequence of $\Upsilon$'s, so that $\| G_x^\cN \( \Upsilon \) \|_{L^2 \( \bT \)} \rightarrow \| G_x^\cN \|$, we get that $\| G_x^\cN \| \leqslant \tfrac{1}{C r^2}$, which completes the proof.
\end{proof}

\bigskip

\section{The main theorem}
\label{sec:main}

We are now ready to state and prove our main theorem.

\begin{theorem}
\label{theorem:nahm}
   The pair $\( \nabla^\cN, \Phi^\cN \)$ is an $\SU (N)$-monopole, that is, it satisfies \cref{eq:mono1,eq:mono_L2}. Furthermore, \cref{eq:mono2,eq:mono3,eq:mono4,eq:mono5} hold and, in some gauge, $\mu$ and $\kappa$ satisfy
   \begin{subequations}
   \begin{align}
      \mu      &= \diag \( \mu_1, \mu_2, \ldots, \mu_n \), \label{eq:mu} \\
      \kappa   &= \diag \( \kappa_1, \kappa_2, \ldots, \kappa_n \), \label{eq:kappa}
   \end{align}
   \end{subequations}
   and for all $a \in \{ 1, 2, \ldots, n \}$, $\mu_a, \kappa_a \in \u \( r_a \)$ satisfy
   \begin{align}
      \mu_a      &= i \lambda_a \id_{\cx^{r_a}}, \label{eq:mu_EVA} \\
      \kappa_a   &= \diag \( i k_{a, 1}^+, i k_{a, 2}^+, \ldots, 0, \ldots, - i k_{a, 2}^-, - i k_{a, 1}^- \). \label{eq:kappa_EVA}
   \end{align}
   Finally, the Yang--Mills--Higgs energy \eqref{eq:YMH_energy} of $\( \nabla^\cN, \Phi^\cN \)$ is given by
   \begin{equation}
      \cE_\YMH \( \nabla^\cN, \Phi^\cN \) = \sum\limits_{a = 1}^n \lambda_a k_a. \label{eq:energy}
   \end{equation}
\end{theorem}

\smallskip

\begin{remark}
   In \cite{charbonneau_nahm_2023}, we also show that the above monopole is without flat factors.
\end{remark}

\smallskip

\begin{proof}[Proof of \Cref{theorem:nahm}]
   First we prove that the pair $\( \nabla^\cN, \Phi^\cN \)$ satisfies \cref{eq:mono1}. We use the notations of \Cref{definition:nahm_construction}. Furthermore, for any $\Psi = \( \uppsi_a \)_{a = 1}^{n - 1} \oplus \( \upphi_b \)_{b = 2}^{n - 1} \in L^2 \( \bT \)$, let us define
   \begin{equation}
      \Sigma_\alpha \( \Psi \) \eqdef \( \( \id_{\cx^{m_a}} \otimes \sigma_\alpha \) \uppsi_a \)_{a = 1}^{n - 1} \oplus \( 0 \)_{b = 2}^{n - 1}.
   \end{equation}
   or, equivalently 
   \begin{equation}
      \Sigma_\alpha \eqdef - \del_\alpha \D_x^\cN. \label{eq:del_alpha_D}
   \end{equation}
   Note that $\Sigma_\alpha$ is unitary, skew-adjoint, and since $H_x^\cN$ commutes with action of quaternions, for all $\alpha \in \{ 1, 2, 3 \}$ we have that
   \begin{equation}
      G_x^\cN \circ \Sigma_\alpha = \Sigma_\alpha \circ G_x^\cN, \label{eq:G_commutes_w_quaternions}
   \end{equation}

   Next we show that
   \begin{equation}
      \Pi_{E_x^\cN} = \id_{L^2 \( \bT \)} - \D_x^\cN \circ G_x^\cN \circ \( \D_x^\cN \)^*.
   \end{equation}
   By construction, $\Pi_{E_x^\cN}$ is an orthogonal projection, that is $\Pi_{E_x^\cN}^2 = \Pi_{E_x^\cN}$ and $\( \Pi_{E_x^\cN} \)^* = \Pi_{E_x^\cN}$. If $\Psi \in E_x^\cN = \ker \( \D_x^\cN \)^*$, then
   \begin{equation}
      \Pi_{E_x^\cN} \( \Psi \) = \Psi - \( \D_x^\cN \circ G_x^\cN \circ \( \D_x^\cN \)^* \) \( \Psi \) = \Psi,
   \end{equation}
   and if $\Psi$ is $L^2$-orthogonal to $E_x^\cN$, then, via a standard usage of the Lax--Milgram theorem, one can show that there exists a (unique) $\Upsilon \in L_1^2 \( \bT \)$, such that $\Psi = \D_x^\cN \( \Upsilon \)$, and thus
   \begin{align}
      \Pi_{E_x^\cN} \( \Psi \)   &= \Psi - \( \D_x^\cN \circ G_x^\cN \circ \( \D_x^\cN \)^* \) \( \Psi \) \\
                                 &= \D_x^\cN \( \Upsilon \) - \( \D_x^\cN \circ G_x^\cN \circ \( \D_x^\cN \)^* \circ \D_x^\cN \) \( \Upsilon \) \\
                                 &= \D_x^\cN \( \Upsilon \) - \( \D_x^\cN \circ G_x^\cN \circ H_x^\cN \) \( \Upsilon \) \\
                                 &= \D_x^\cN \( \Upsilon \) - \D_x^\cN \( \Upsilon \) \\
                                 &= 0.
   \end{align}
   Hence, $\Pi_{E_x^\cN}$ is indeed the $L^2$-orthogonal projection onto $\Pi_{E_x^\cN}$.

   Note that the connection $\nabla^\cN$ is a Berry connection, in the sense of \cite{avron_chern_1989}. Thus by \cite{avron_chern_1989}*{Equation~(1.9)}, we get that
   \begin{equation}
      F_{\nabla^\cN} = \Pi_{E_x^\cN} \circ \( \rd \Pi_{E_x^\cN} \wedge \rd \Pi_{E_x^\cN} \) \circ \Pi_{E_x^\cN}.
   \end{equation}
   Equivalently, if we write
   \begin{equation}
      F_{\nabla^\cN} = \tfrac{1}{2} \sum\limits_{\alpha, \beta = 1}^3 F_{\alpha, \beta} \rd x^\alpha \wedge \rd x^\beta,
   \end{equation}
   then for all $\alpha, \beta \in \{ 1, 2, 3 \}$, we have
   \begin{equation}
      F_{\alpha, \beta} = \Pi_{E_x^\cN} \circ \( \del_\alpha \Pi_{E_x^\cN} \) \circ \( \del_\beta \Pi_{E_x^\cN} \) \circ \Pi_{E_x^\cN} - \Pi_{E_x^\cN} \circ \( \del_\beta \Pi_{E_x^\cN} \) \circ \( \del_\alpha \Pi_{E_x^\cN} \) \circ \Pi_{E_x^\cN}.
   \end{equation}
   Using \cref{eq:del_alpha_D}, we first compute $\del_\alpha \Pi_{E_x^\cN}$
   \begin{align}
      \del_\alpha \Pi_{E_x^\cN}  &= - \del_\alpha \( \D_x^\cN \circ G_x^\cN \circ \( \D_x^\cN \)^* \) \\
                                 &= - \( \del_\alpha \D_x^\cN \) \circ G_x^\cN \circ \( \D_x^\cN \)^* - \D_x^\cN \circ \( \del_\alpha G_x^\cN \) \circ \( \D_x^\cN \)^* - \D_x^\cN \circ G_x^\cN \circ \( \del_\alpha \D_x^\cN \)^*.
   \end{align}
   Using
   \begin{equation}
      \Pi_{E_x^\cN} \circ \D_x^\cN = 0 = \( \D_x^\cN \)^* \circ\Pi_{E_x^\cN}, \label{eq:piD}
   \end{equation}
   and \cref{eq:del_alpha_D}, we get
   \begin{subequations}
   \begin{align}
      \Pi_{E_x^\cN} \circ \( \del_\alpha \Pi_{E_x^\cN} \)  &= \Pi_{E_x^\cN} \circ \Sigma_\alpha \circ G_x^\cN \circ \( \D_x^\cN \)^*, \label{eq:Pi_d_alpha_Pi} \\
      \( \del_\beta \Pi_{E_x^\cN} \) \circ \Pi_{E_x^\cN}   &= - \D_x^\cN \circ G_x^\cN \circ \Sigma_\beta \circ \Pi_{E_x^\cN}. \label{eq:d_beta_Pi_Pi}
   \end{align}
   \end{subequations}
   Combining \cref{eq:G_commutes_w_quaternions,eq:piD,eq:Pi_d_alpha_Pi,eq:d_beta_Pi_Pi} yields
   \begin{align}
      \Pi_{E_x^\cN} \circ \( \del_\alpha \Pi_{E_x^\cN} \) \circ \( \del_\beta \Pi_{E_x^\cN} \) \circ \Pi_{E_x^\cN} &= - \Pi_{E_x^\cN} \circ \Sigma_\alpha \circ G_x^\cN \circ \( \D_x^\cN \)^* \circ \D_x^\cN \circ G_x^\cN \circ \Sigma_\beta \circ \Pi_{E_x^\cN} \\
         &= - \Pi_{E_x^\cN} \circ \Sigma_\alpha \circ G_x^\cN \circ H_x^\cN \circ G_x^\cN \circ \Sigma_\beta \circ \Pi_{E_x^\cN} \\
         &= - \Pi_{E_x^\cN} \circ \Sigma_\alpha \circ G_x^\cN \circ \Sigma_\beta \circ \Pi_{E_x^\cN} \\
         &= - \Pi_{E_x^\cN} \circ \Sigma_\alpha \circ \Sigma_\beta \circ G_x^\cN \circ \Pi_{E_x^\cN} \\
         &= \delta_{\alpha, \beta} \Pi_{E_x^\cN} \circ G_x^\cN \circ \Pi_{E_x^\cN} - \sum\limits_{\gamma = 1}^3 \epsilon_{\alpha, \beta, \gamma} \Pi_{E_x^\cN} \circ \Sigma_\gamma \circ G_x^\cN \circ \Pi_{E_x^\cN}.
   \end{align}
   In the last step we use the commutation relations of quaternions. Hence, we get
   \begin{equation}
      F_{\alpha, \beta} = - 2 \sum\limits_{\gamma = 1}^3 \epsilon_{\alpha, \beta, \gamma} \Pi_{E_x^\cN} \circ \Sigma_\gamma \circ G_x^\cN \circ \Pi_{E_x^\cN}. \label{eq:F_formula}
   \end{equation}

   Similarly, we can compute the components of $\nabla^\cN \Phi^\cN$. Let $M \colon L^2 \( \bT \) \rightarrow L^2 \( \bT \)$ be as in \cref{eq:M_def}. Thus, $\Phi^\cN = \Pi_{E_x^\cN} \circ M$. Note that $M$ is independent of $x$. Fix $\alpha \in \{ 1, 2, 3 \}$. Then, using \cref{eq:piD,eq:Pi_d_alpha_Pi,eq:d_beta_Pi_Pi} again, we get
   \begin{align}
      \nabla_\alpha^\cN \Phi^\cN &= \( \Pi_{E_x^\cN} \circ \del_\alpha \) \circ \( \Pi_{E_x^\cN} \circ M \) \circ \Pi_{E_x^\cN} - \Phi^\cN \circ \del_\alpha \\
         &= \Pi_{E_x^\cN} \circ \( \del_\alpha \Pi_{E_x^\cN} \) \circ M \circ \Pi_{E_x^\cN} + \Pi_{E_x^\cN} \circ M \circ \( \del_\alpha \Pi_{E_x^\cN} \) \circ \Pi_{E_x^\cN} \\
         &= \Pi_{E_x^\cN} \circ \Sigma_\alpha \circ G_x^\cN \circ \( \D_x^\cN \)^* \circ M \circ \Pi_{E_x^\cN} - \Pi_{E_x^\cN} \circ M \circ \D_x^\cN \circ G_x^\cN \circ \Sigma_\alpha \circ \Pi_{E_x^\cN} \\
         &= \Pi_{E_x^\cN} \circ \Sigma_\alpha \circ G_x^\cN \circ \left[ \( \D_x^\cN \)^*, M \right] \circ \Pi_{E_x^\cN} + \Pi_{E_x^\cN} \circ \left[ \D_x^\cN, M \right] \circ \Sigma_\alpha \circ G_x^\cN \circ \Pi_{E_x^\cN}.
   \end{align}
   Using the definitions of $\D_x^\cN$ and $M$, we get
   \begin{equation}
      \left[ \( \D_x^\cN \)^*,M\right] \circ \Pi_{E_x^\cN} = \Pi_{E_x^\cN} \circ \left[ \D_x^\cN,M\right] = - \Pi_{E_x^\cN},
   \end{equation}
   and thus
   \begin{equation}
      \nabla_\alpha^\cN \Phi^\cN = - 2 \: \Pi_{E_x^\cN} \circ \Sigma_\alpha \circ G_x^\cN \circ \Pi_{E_x^\cN}. \label{eq:dPhi_formula}
   \end{equation}
   Now \cref{eq:F_formula,eq:dPhi_formula} imply that the pair $\( \nabla^\cN, \Phi^\cN \)$ satisfies the monopole \cref{eq:mono1}. Clearly, $\( \nabla^\cN, \Phi^\cN \)$ is compatible with the Hermitian structure of $E^\cN$, thus it is an $\rU (N)$-monopole.

   Next we show that $\( \nabla^\cN, \Phi^\cN \)$ has the correct asymptotics, that is, in some gauge it satisfies \cref{eq:mono2,eq:mono3,eq:mono4,eq:mono5} with \cref{eq:mu,eq:kappa,eq:mu_EVA,eq:kappa_EVA}. These in turn imply finite energy \cref{eq:mono_L2} with \cref{eq:energy}. Since $\mu$ is traceless by \cref{eq:mu,eq:mu_EVA} and condition (2) in \Cref{definition:type}, we have by \Cref{remark:U(N)} that $\( \nabla^\cN, \Phi^\cN \)$ is, in fact, an $\SU (N)$-monopole. We prove these equations and conditions by constructing a frame of $E^\cN|_{\rl^3 - \left\{ \textnormal{o} \right\}}$ in which they are apparent. We remark here that similar (approximate) solutions were used in \cite{Cherkis-Larrain-Stern-2}.

   Let
   \begin{equation}
      \delta \eqdef \tfrac{1}{4} \min \( \left\{ \lambda_2 - \lambda_1, \lambda_3 - \lambda_2, \ldots, \lambda_n - \lambda_{n - 1} \right\} \) \in \rl_+, \label{eq:EVA_gap}
   \end{equation}
   and let $\chi : \rl_+ \rightarrow [0, 1]$ be a smooth function such that $\chi$ is identically $1$ on $\( 0, \delta \)$, identically zero on $\( 2 \delta, \infty\)$, and $|\dot{\chi}| = - \dot{\chi} \leqslant 2 \delta^{- 1}$ everywhere.

   Fix $x \in \rl^3 - \left\{ \textnormal{o} \right\}$. For each $a \in \{ 1, 2, \ldots, n - 1 \}$ and $b \in \{ 1, 2, \ldots, r_a^- \}$, let $u_{a, b}^- (x) \in X_a^-$ be the vector $u(k^{-}_{a,b},x)$ of \Cref{lemma:simultaneous_EVEs} in the summand $R_{k^-_{a,b}}$ of $X_a^-$ imposed by the decomposition  \cref{eq:reduction}, and $\widetilde{\Psi}_{a, b}^- (x) = \( \uppsi_{a^\prime} \)_{a^\prime = 1}^{n - 1} \oplus \( 0 \)_{b^\prime = 2}^{n - 1} \in L^2 \( \bT \)$ be so that $\uppsi_{a^\prime} = 0$, if $a^\prime \neq a$, for all $t \in \( \lambda_a , \lambda_{a + 1} \)$, we have, for some normalization constant $c_r=1+O(e^{-\delta r})$, 
   \begin{equation}
      \uppsi_a (t) = c_r\chi \( t - \lambda_a \) \tfrac{\( 2 r \)^{\frac{k_{a, b}^-}{2}}}{\sqrt{\( k_{a, b}^- - 1 \)!}} \( t - \lambda_a \)^{\frac{k_{a, b}^- - 1}{2}} e^{- r \( t - \lambda_a \)} u_{a, b}^- (x).
   \end{equation}

   Similarly, for each $a \in \{ 1, 2, \ldots, n - 1 \}$ and $b \in \{ 1, 2, \ldots, r_{a + 1}^+ \}$, let $u_{a + 1, b}^+ (x) \in X_a^+$ be the vector $u(-k^+_{a+1,b},x)$ of \Cref{remark:simultaneous_EVEsnegativek} in the summand $R_{k^+_{a+1,b}}$ of $X_a^+$, and $\widetilde{\Psi}_{a, b}^+ = \( \uppsi_{a^\prime} \)_{a^\prime = 1}^{n - 1} \oplus \( 0 \)_{b^\prime = 2}^{n - 1} \in L^2 \( \bT \)$ be so that $\uppsi_{a^\prime} = 0$, if $a^\prime \neq a$ and for all $t \in \( \lambda_a, \lambda_{a + 1} \)$, we have, for some normalization constant $c_r=1+O(e^{-\delta r})$, 
   \begin{equation}
      \uppsi_a (t) = c_r\chi \( \lambda_{a + 1} - t \) \tfrac{\( 2 r \)^{\frac{k_{a + 1, b}^+}{2}}}{\sqrt{\( k_{a + 1, b}^+ - 1 \)!}} \( \lambda_{a + 1} - t \)^{\frac{k_{a + 1, b}^+ - 1}{2}} e^{- r \( \lambda_{a + 1} - t \)} u_{a + 1, b}^+ (x).
   \end{equation}

   Finally, let $P_x^\pm$ be the projection onto the $\( \pm i r \)$-eigenspaces of $\widehat{x}$. For each $a \in \{ 2, 3, \ldots, n - 1 \}$ and $b \in \{ 1, 2, \ldots, r_a^0 \}$, let $\upxi_{a, b} \in Y_{a - 1}^+$ as in \cref{eq:C_a_def}, and $\widetilde{\Psi}_{a, b}^0 = \( \uppsi_{a^\prime} \)_{a^\prime = 1}^{n - 1} \oplus \( \upphi_{b^\prime} \)_{b^\prime = 2}^{n - 1} \in L^2 \( \bT \)$ be so that $\uppsi_{a^\prime} = 0$, if $a^\prime \notin \{ a - 1, a \}$, and, for some normalization constant $c_r=1+O(r^{-1})$, 
   \begin{equation}
      \begin{array}{rll}
         \uppsi_{a - 1} (t)   =& c_r\chi \( \lambda_a - t \) e^{- r \( \lambda_a - t \)} \( \id_{V_{a-1}^+}  \otimes P_x^- \) \( \upxi_{a, b} \),         & \mbox{if } t \in \( \lambda_{a - 1}, \lambda_a \), \\
         \uppsi_a (t)         =& - c_r\chi \( t - \lambda_a \) e^{- r \( t - \lambda_a \)} \( C_a \otimes P_x^+ \) \( \upxi_{a, b} \), & \mbox{if } t \in \( \lambda_a, \lambda_{a + 1} \),
      \end{array}
   \end{equation}
   and $\upphi_{b^\prime} = 0$, unless $b^\prime = a$, in which case it is given by \cref{eq:BVP_a_+,eq:BVP_a_-}, that is
   \begin{equation}
      \upphi_a = - i c_r\tfrac{\upxi_{a, b}}{\left| \upxi_{a, b} \right|}. \label{eq:upphi_a_def}
   \end{equation}

   Now for all relevant $\( a, b, \s \), \( a^\prime, b^\prime, \s^\prime \)$, we now have that
   \begin{equation}
      \left\langle \widetilde{\Psi}_{a, b}^\s (x) \middle| \widetilde{\Psi}_{a^\prime, b^\prime}^{\s^\prime} (x) \right\rangle_{L^2 \( \bT \)} = \delta_{\( a, b, \s \), \( a^\prime, b^\prime, \s^\prime \)} . \label{eq:tilde_Psi_brakets}
   \end{equation}
   Let  
   \begin{equation}
   	\Psi_{a, b}^\s (x) \eqdef \Pi_{E_x^\cN} \( \widetilde{\Psi}_{a, b}^\s (x) \) = \widetilde{\Psi}_{a, b}^\s (x) - \D_x^\cN \( \varphi_{a, b}^\s (x) \).
   \end{equation}
   By design, the action of $\( \D_x^\cN \)^*$ on $\widetilde{\Psi}_{a, b}^\s$ differs from the model \cref{eq:model_eq} only by a bounded, $x$-independent endomorphism of $\( \lambda_a, \lambda_{a + 1} \) \times \( \cx^{m_a} \otimes \bH \)$. Thus, the family
   \begin{equation}
      \left\{ \( \D_x^\cN \)^* \( \widetilde{\Psi}_{a, b}^\s (x) \) \middle| x \in \rl^3 \right\} \subset L^2 \( \bT \)
   \end{equation}
   is bounded, $x$-independently. Let
   \begin{equation}
      \varphi_{a, b}^\s (x) \eqdef G_x^\cN \( \( \D_x^\cN \)^* \( \widetilde{\Psi}_{a, b}^\s (x) \) \).
   \end{equation}
   Then $\varphi_{a, b}^\s (x) \in L_1^2 \( \bT \)$ and
   \begin{equation}
   	\Psi_{a, b}^\s (x) = \widetilde{\Psi}_{a, b}^\s (x) - \D_x^\cN \( \varphi_{a, b}^\s (x) \).
   \end{equation} 
   By \Cref{lemma:Greens_bound}, we have $\| \varphi_{a, b}^\s (x) \|_{L^2 \( \bT \)} = O \( r^{- 2} \)$, and
   \begin{equation}
      \| \D_x^\cN \( \varphi_{a, b}^\s (x) \) \|_{L^2 \( \bT \)}^2 = \left\langle \( \D_x^\cN \)^* \( \widetilde{\Psi}_{a, b}^\s (x) \) \middle| \varphi_{a, b}^\s (x) \right\rangle_{L^2 \( \bT \)} = O \( r^{- 2} \). \label{eq:D_phi_bound}
   \end{equation}
   Thus, by \cref{eq:tilde_Psi_brakets,eq:D_phi_bound}, we get   
   \begin{equation}
      \left\langle \Psi_{a, b}^\s (x) \middle| \Psi_{a^\prime, b^\prime}^{\s^\prime} (x) \right\rangle_{L^2 \( \bT \)} = \delta_{\( a, b, \s \), \( a^\prime, b^\prime, \s^\prime \)} + O \( r^{- 2} \). \label{eq:Psi_brakets}
   \end{equation}

   For all relevant $a, b$, let $u_{a, b}^0 \eqdef \tfrac{\upxi_{a, b}}{\left| \upxi_{a, b} \right|}$ and note that it is independent of $x$. While $u_{a, b}^\s (x)$, and thus $\Psi_{a, b}^\s (x)$, are both ambiguous up to a $\rU (1)$ factor, the complex lines $\( \cL_{a, b}^\s \)_x \eqdef \spn_\cx \( u_{a, b}^\s \)$ and $\( L_{a, b}^\s \)_x \eqdef \spn_\cx \( \left\{ \Psi_{a, b}^\s (x) \right\} \)$ are well-defined and $\cL_{a, b}^\s$ and $L_{a, b}^\s$ are a smooth, complex line bundle over $\rl^3 - \left\{ \textnormal{o} \right\}$, that are canonically isomorphic via $u_{a, b}^\s (x) \mapsto \Psi_{a, b}^\s (x)$. Note that this is not an isomorphism of Hermitian structures, albeit in some sense close to one by \cref{eq:Psi_brakets}. By \Cref{lemma:model_solution_k_not_zero,remark:Dirac_mono}, the Chern number of $\cL_{a, b}^\s$, and hence of $L_{a, b}^\s$, is $\s k_{a, b}^\s$. Since $u_{a, b}^\s (x)$ is independent of $r$, $\cL_{a, b}^\s$ can also be viewed as a bundle over $S_\infty^2$. Let its Chern connection be $\nabla^{\cL_{a, b}^\s}$. Now we can finally define
   \begin{equation}
      \( E^\infty, \nabla^\infty \) \eqdef \bigoplus_{a = 1}^{n - 1} \bigoplus_{\s \in \{ - 1, 0, 1 \}} \bigoplus_{b = 1}^{r_a^\s} \( \cL_{a, b}^\s, \nabla^{\cL_{a, b}^\s} \).
   \end{equation}
   Note that we also have
   \begin{equation}
      E^\cN|_{\rl^3 - \left\{ \textnormal{o} \right\}} \cong \bigoplus_{a = 1}^{n - 1} \bigoplus_{\s \in \{ - 1, 0, 1 \}} \bigoplus_{b = 1}^{r_a^\s} L_{a, b}^\s.
   \end{equation}
   Thus, we constructed an asymptotically orthogonal decomposition of $E^\cN |_{\rl^3 - \left\{ \textnormal{o} \right\}}$ into Hermitian line bundles together with an asymptotically unitary isomorphism of $E^\cN|_{\rl^3 - \left\{ \textnormal{o} \right\}}$ and $\pi^* \( E^\infty \)$. Through this isomorphism  we claim that \cref{eq:mono2,eq:mono3,eq:mono4,eq:mono5,eq:mu,eq:kappa,eq:mu_EVA,eq:kappa_EVA} hold.

   Let us compute 
   \begin{align}
      \left\langle \Psi_{a, b}^\s (x) \middle| \Phi^\cN \( \Psi_{a^\prime, b^\prime}^{\s^\prime} (x) \) \right\rangle_{L^2 \( \bT \)} &= \left\langle \widetilde{\Psi}_{a, b}^\s (x) \middle| M \( \widetilde{\Psi}_{a^\prime, b^\prime}^{\s^\prime} (x) \) \right\rangle_{L^2 \( \bT \)} \\
         & \quad - \left\langle \widetilde{\Psi}_{a, b}^\s (x) \middle| M \( \D_x^\cN \( \varphi_{a^\prime, b^\prime}^{\s^\prime} (x) \) \) \right\rangle_{L^2 \( \bT \)} \\
         & \quad + \left\langle M \( \D_x^\cN \( \varphi_{a, b}^\s (x) \) \) \middle| \widetilde{\Psi}_{a^\prime, b^\prime}^{\s^\prime} (x) \right\rangle_{L^2 \( \bT \)} \\
         & \quad + \left\langle \D_x^\cN \( \varphi_{a, b}^\s (x) \) \middle| M \( \D_x^\cN \( \varphi_{a^\prime, b^\prime}^{\s^\prime} (x) \) \) \right\rangle_{L^2 \( \bT \)} \\
         &= \left\langle \widetilde{\Psi}_{a, b}^\s (x) \middle| M \( \widetilde{\Psi}_{a^\prime, b^\prime}^{\s^\prime} (x) \) \right\rangle_{L^2 \( \bT \)} \\
         & \quad - \left\langle \tilde\Psi_{a, b}^\s (x) \middle| \left[ M, \D_x^\cN \right] \( \varphi_{a^\prime, b^\prime}^{\s^\prime} (x) \) \right\rangle_{L^2 \( \bT \)} \\
         & \quad - \left\langle \( \D_x^\cN \)^* \( \widetilde{\Psi}_{a, b}^\s (x) \) \middle| M\varphi_{a^\prime, b^\prime}^{\s^\prime} (x) \right\rangle_{L^2 \( \bT \)} \\
         & \quad + \left\langle \left[ M, \D_x^\cN \right] \( \varphi_{a, b}^\s (x) \) \middle| \widetilde{\Psi}_{a^\prime, b^\prime}^{\s^\prime} (x) \right\rangle_{L^2 \( \bT \)} \\
         & \quad +\left\langle M\varphi_{a, b}^\s (x) \middle| \( \D_x^\cN \)^* \( \widetilde{\Psi}_{a^\prime, b^\prime}^{\s^\prime} (x) \) \right\rangle_{L^2 \( \bT \)} \\
         & \quad + O \( r^{- 2} \) \\
         &= \left\langle \widetilde{\Psi}_{a, b}^\s (x) \middle| M \( \widetilde{\Psi}_{a^\prime, b^\prime}^{\s^\prime} (x) \) \right\rangle_{L^2 \( \bT \)} + O \( r^{- 2} \).
   \end{align}
   Note that $\left\langle \widetilde{\Psi}_{a, b}^\s (x) \middle| M \( \widetilde{\Psi}_{a^\prime, b^\prime}^{\s^\prime} (x) \) \right\rangle_{L^2 \( \bT \)} = 0$, unless $\( a, b, \s \) = \( a^\prime, b^\prime, \s^\prime \)$. When $\( a, b, \s \) = \( a^\prime, b^\prime, \s^\prime \)$, $\s = \s^\prime = -$ and $k_{a, b}^- \in \Z_+$, we get 
   \begin{align}
      \left\langle \Psi_{a, b}^- (x) \middle| \Phi^\cN \( \Psi_{a, b}^- (x) \) \right\rangle_{L^2 \( \bT \)} &= \left\langle \widetilde{\Psi}_{a, b}^- (x) \middle| M \( \widetilde{\Psi}_{a, b}^- (x) \) \right\rangle_{L^2 \( \bT \)} + O \( r^{- 2} \) \\
         &= i c_r^2\int\limits_{\lambda_a}^{\lambda_{a + 1}} t \chi \( t - \lambda_a \)^2 \tfrac{\( 2 r \)^{k_{a, b}^-}}{\( k_{a, b}^- - 1 \)!} \( t - \lambda_a \)^{k_{a, b}^- - 1} e^{- 2 r \( t - \lambda_a \)} \rd t + O \( r^{- 2} \) \\
         &= i c_r^2\tfrac{\( 2 r \)^{k_{a, b}^-}}{\( k_{a, b}^- - 1 \)!} \int\limits_0^{\lambda_{a + 1} - \lambda_a} \( s + \lambda_a \) \chi \( s \)^2 s^{k_{a, b}^- - 1} e^{- 2 r s} \rd s + O \( r^{- 2} \) \\
         &= i \( \lambda_a - \tfrac{- k_{a, b}^-}{2 r} \) + O \( r^{- 2} \).
   \end{align}
   Analogous computations yield the asymptotics for $\s = +$. When $\s = 0$, we get
   \begin{align}
      \left\langle \Psi_{a, b}^0 (x) \middle| \Phi^\cN \( \Psi_{a, b}^0 (x) - i \lambda_a \Psi_{a, b}^0 (x) \) \right\rangle_{L^2 \( \bT \)} &= \left\langle \widetilde{\Psi}_{a, b}^0 (x) \middle| \( M - i \lambda_a \) \( \widetilde{\Psi}_{a, b}^0 (x) \) \right\rangle_{L^2 \( \bT \)} \\
         & \quad + O \( r^{- 2} \) \\
         &= i \int\limits_{\lambda_{a - 1}}^{\lambda_a} t \chi \( \lambda_a - t \)^2 e^{- 2 r \( \lambda_a - t \)} \left| \(  \id_{V_{a-1}^+}\otimes P_x^- \) \( \upxi_{a, b} \) \right|^2 \rd t \\
         & \quad + i \int\limits_{\lambda_a}^{\lambda_{a + 1}} t \chi \( t - \lambda_a \)^2 e^{- 2 r \( t - \lambda_a \)} \left| \( C_a\otimes P_x^+  \) \( \upxi_{a, b} \) \right|^2 \rd t \\
         & \quad + O \( r^{- 2} \) \\
         &= i \tfrac{\left| \( C_a^* \otimes P_x^- \) \( \upxi_{a, b} \) \right|^2  - \left| \( \id_{V_a^-} \otimes P_x^+ \) \( \upxi_{a, b} \) \right|^2}{4 r^2} \int\limits_0^\infty s \chi \( s \)^2 e^{- s} \rd s \\
         & \quad + O \( r^{- 2} \) \\
         &= O \( r^{- 2} \).
   \end{align}

   In summary, we get, for all relevant $\( a, b, \s \)$, that
   \begin{equation}
      \Phi^\cN \( \Psi_{a, b}^\s (x) \) = i \( \lambda_a - \tfrac{\s k_{a, b}^\s}{2 r} \) \Psi_{a, b}^\s (x) + O \( r^{- 2} \). \label{eq:Phi_Psi_EVE}
   \end{equation}
   
   Thus we proved \cref{eq:mono3} and the \cref{eq:mu,eq:kappa,eq:mu_EVA,eq:kappa_EVA} for the Higgs field (but not yet for its derivative, that is, not for \cref{eq:mono4}). Furthermore, given \cref{eq:energymuk}, we now proved \cref{eq:energy}. 

   Next we prove \cref{eq:mono2}. When $\s = 0$, then, using that $\upxi_{a, b}$ is $x$-independent, we have that
   \begin{equation}
      \| \del_r \widetilde{\Psi}_{a, b}^0 (x) \|_{L^2 \( \bT \)} = O \( r^{- 3} \),
   \end{equation}
   which implies $\| \nabla_{\del_r}^\cN \Psi_{a, b}^0 (x) \|_{L^2 \( \bT \)} = O \( r^{- 3} \)$. Thus, the radial part of \cref{eq:mono2} is verified for $\s = 0$.
   
   When $\s = \pm$, we choose the vector $u_{a, b}^\s (x)$ to be $r$-independent. In this case we have
   \begin{equation}
      \del_r \widetilde{\Psi}_{a, b}^\s (x) = - \s \( \lambda_a - \tfrac{\s k_{a, b}^\s}{2 r} - t \) \widetilde{\Psi}_{a, b}^\s (x).
   \end{equation}
   Now using \cref{eq:Pi_d_alpha_Pi,eq:Phi_Psi_EVE}, we get
   \begin{align}
      \nabla_{\del_r}^\cN \Psi_{a, b}^\s (x) &= \Pi_{E_x^\cN} \( \del_r \( \Pi_{E_x^\cN} \widetilde{\Psi}_{a, b}^\s (x) \) \) \\
         &= \Pi_{E_x^\cN} \( \del_r \widetilde{\Psi}_{a, b}^\s (x) \) + \( \Pi_{E_x^\cN} \circ \( \del_r \Pi_{E_x^\cN} \) \) \( \widetilde{\Psi}_{a, b}^\s (x) \) \\
         &= \Pi_{E_x^\cN} \( - \s \( \lambda_a - \tfrac{\s k_{a, b}^-}{2 r} - t \) \widetilde{\Psi}_{a, b}^\s (x) \) + \sum\limits_{\alpha = 1}^3 \tfrac{x_\alpha}{r} \( \Pi_{E_x^\cN} \circ \( \del_\alpha \Pi_{E_x^\cN} \) \) \( \widetilde{\Psi}_{a, b}^\s (x) \)+O(e^{-r\delta})  \\
         &= \Pi_{E_x^\cN} \( - \s \( \lambda_a - \tfrac{\s k_{a, b}^-}{2 r} - t \) \widetilde{\Psi}_{a, b}^\s (x) \) + \sum\limits_{\alpha = 1}^3 \tfrac{x_\alpha}{r} \( \Pi_{E_x^\cN} \circ \Sigma_\alpha \circ G_x^\cN \circ \( \D_x^\cN \)^* \) \( \widetilde{\Psi}_{a, b}^\s (x) \) +O(e^{-r\delta})  \\
         &= O \( r^{- 2} \).
   \end{align}
   Thus, the radial part of \cref{eq:mono2} is verified for $\s = \pm$.

   Next we prove \cref{eq:mono2} for directions perpendicular to the radial one. If $v \in T_x \S_r^2$, then $\nabla_v^\infty \Psi_{a, b}^\s (x) = \Pi_{E_x^\cN} \del_v \widetilde{\Psi}_{a, b}^\s (x)$.  Thus, using \cref{eq:Pi_d_alpha_Pi}, we get
   \begin{align}
      \( \nabla_v^\cN - \nabla_v^\infty \) \Psi_{a, b}^\s (x)  &= \Pi_{E_x^\cN} \circ \( \del_v \Pi_{E_x^\cN} \) \widetilde{\Psi}_{a, b}^\s (x) \\
      &= \sum\limits_{\alpha = 1}^3 v_\alpha \Pi_{E_x^\cN} \circ \Sigma_\alpha \circ G_x^\cN \circ \( \D_x^\cN \)^* \widetilde{\Psi}_{a, b}^\s (x)  \label{eq:R1_formula} \\
      &= O \( |v| r^{- 2} \),
   \end{align}
   and thus in this frame \cref{eq:mono2} holds.

   Next we verify \cref{eq:mono4}. For each $x \in \rl^3 - \left\{ \textnormal{o} \right\}$, let $\Sigma (x) \eqdef \sum\limits_{\alpha = 1}^3 \tfrac{x_\alpha}{r} \Sigma_\alpha$. By design
   \begin{equation}
      \Sigma (x) \( \widetilde{\Psi}_{a^\prime, b^\prime}^{\s^\prime} (x) \) = -\s^\prime i \widetilde{\Psi}_{a^\prime, b^\prime}^{\s^\prime} (x). \label{eq:sigma_psi_EVE}
   \end{equation}
   Using \cref{eq:Greens_bound,eq:dPhi_formula,eq:D_phi_bound,eq:sigma_psi_EVE}, we get
   \begin{align}
      \left\langle \Psi_{a, b}^\s (x) \middle| \( \nabla_{\del_r}^\cN \Phi^\cN \) \( \Psi_{a^\prime, b^\prime}^{\s^\prime} (x) \) \right\rangle_{L^2 \( \bT \)} &= -2 \s^\prime i \left\langle \widetilde{\Psi}_{a, b}^\s (x) \middle| G_x^\cN \( \widetilde{\Psi}_{a^\prime, b^\prime}^{\s^\prime} (x) \) \right\rangle_{L^2 \( \bT \)} + O \( r^{- 3} \).
   \end{align}

   Next we give sharp enough approximations for $G_x^\cN \( \widetilde{\Psi}_{a^\prime, b^\prime}^{\s^\prime} (x) \) \in L_1^2 \( \bT \)$. Let $\delta > 0$ as in \cref{eq:EVA_gap} and $\chi : \rl_+ \rightarrow [0, 1]$ as defined in the sentence after \cref{eq:EVA_gap}. When $\s^\prime = -$, let us define $\Upsilon_{a^\prime, b^\prime}^- (x) = \( \upsilon_c \)_{c = 1}^{n - 1} \in L_1^2 \( \bT \)$ so that for any $c \in \{ 1, 2, \ldots, n - 1 \} - \{ a^\prime \}$, $\upsilon_c = 0$, and for all $t \in \( \lambda_{a^\prime}, \lambda_{a^\prime + 1} \)$ we have
   \begin{equation}
      \upsilon_{a^\prime} (t) = \chi \( t - \lambda_{a^\prime} \) \tfrac{2}{k_{a^\prime, b^\prime}^- + 1} \tfrac{\( 2 r \)^{\frac{k_{a^\prime, b^\prime}^-}{2} - 1}}{\sqrt{\( k_{a^\prime, b^\prime}^- - 1 \)!}} \( t - \lambda_{a^\prime} \)^{\frac{k_{a^\prime, b^\prime}^- + 1}{2}} e^{- r \( t - \lambda_{a^\prime} \)} u_{a^\prime, b^\prime}^- (x).
   \end{equation}
   Now simple computation yields that
   \begin{align}
      \| \Upsilon_{a^\prime, b^\prime}^- (x) \|_{L^2 \( \bT \)}^2 &= \tfrac{k_{a^\prime, b^\prime}^-}{4 \( k_{a^\prime, b^\prime}^- + 1 \) r^4} + O \( r^{- 6} \), \\
      \left\langle \widetilde{\Psi}_{a, b}^\s (x) \middle| \Upsilon_{a^\prime, b^\prime}^- (x) \right\rangle_{L^2 \( \bT \)}  &= \delta_{\( a, b, \s \), \( a^\prime, b^\prime, - \)} \tfrac{k_{a^\prime, b^\prime}^-}{2 \( k_{a^\prime, b^\prime}^- + 1 \) r^2} + O \( r^{- 3} \).
   \end{align}
   Moreover, $\Upsilon_{a^\prime, b^\prime}^-$ was chosen so that
   \begin{equation}
      H_x^\cN \( G_x^\cN \( \widetilde{\Psi}_{a^\prime, b^\prime}^{\s^\prime} (x) \) - \Upsilon_{a^\prime, b^\prime}^- (x) \) = \cR \( \Upsilon_{a^\prime, b^\prime}^- (x) \),
   \end{equation}
   where $\cR$ is a bounded endomorphism over $\( \lambda_a, \lambda_{a + 1} \) \times \( \cx^{m_a} \otimes \bH \)$. Thus,
   \begin{equation}
      \| G_x^\cN \( \widetilde{\Psi}_{a^\prime, b^\prime}^{\s^\prime} (x) \) - \Upsilon_{a^\prime, b^\prime}^- (x) \|_{L^2 \( \bT \)} =  O \( r^{- 3} \).
   \end{equation}
   Now we have that
   \begin{align}
      \left\langle \Psi_{a, b}^\s (x) \middle| \( \nabla_{\del_r}^\cN \Phi^\cN \) \( \Psi_{a^\prime, b^\prime}^- (x) \) \right\rangle_{L^2 \( \bT \)} &= - 2 i \left\langle \widetilde{\Psi}_{a, b}^\s (x) \middle| G_x^\cN \( \widetilde{\Psi}_{a^\prime, b^\prime}^- (x) \) \right\rangle_{L^2 \( \bT \)} + O \( r^{- 3} \) \\
         &= - 2 i \left\langle \widetilde{\Psi}_{a, b}^\s (x) \middle| \Upsilon_{a^\prime, b^\prime}^- (x) \right\rangle_{L^2 \( \bT \)} + O \( r^{- 3} \) \\
         &= - 2 i \delta_{\( a, b, \s \), \( a^\prime, b^\prime, - \)} \tfrac{k_{a^\prime, b^\prime}^-}{4 r^2} + O \( r^{- 3} \) \\
         &= \delta_{\( a, b, \s \), \( a^\prime, b^\prime, - \)} i \tfrac{- k_{a^\prime, b^\prime}^-}{2 r^2} + O \( r^{- 3} \).
   \end{align}
   Similar computations yield the same asymptotics for $\s \in \{ 0, + \}$, which proves \cref{eq:mono4}.

   The last thing to verify is \cref{eq:mono5}. Let us recall that the remainder terms, defined in the frame above, are as
   \begin{align}
      \cR_1 (x)   &\eqdef r^2 \( \nabla^\cN - \del_r \otimes \rd r - \nabla^\infty \), \\
      \cR_2 (x)   &\eqdef r^2 \( \Phi^\cN - \mu + \tfrac{1}{2 r} \kappa \), \\
      \cR_3 (x)   &\eqdef r^3 \( \nabla^\cN \Phi^\cN - \tfrac{1}{2 r^2} \kappa \otimes \rd r \).
   \end{align}
   We have, in fact, already proved that they are all in $L^\infty \( \rl^3 \)$. Thus, the only thing left to show is that for all $i \in \{ 1, 2, 3 \}$ we have that $\nabla^\cN \cR_i \in L^\infty \( \rl^3 \)$. Since we have already verified \cref{eq:mono1,eq:mu_parallel,eq:kappa_parallel,eq:curvature_at_infinity}, it is now enough to show only that $\nabla^{\cN \otimes \mathrm{LC}} \cR_1 \in L^\infty \( \rl^3 \)$, where $\nabla^{\cN \otimes \mathrm{LC}}$ is the connection induced by $\nabla^\cN$ and the Levi-Civita connection of $\rl^3$. First, let us pick a coordinate system so that $\del_1 = \del_r$ at $x$. Then, for all relevant $\( a, b, \s \)$, and $\alpha, \beta$, we have
   \begin{align}
      \( \nabla_{\del_\alpha}^{\cN \otimes \mathrm{LC}} \cR_1 \) \( \del_\beta \) \( \Psi_{a, b}^\s  \)   &= \nabla_{\del_\alpha}^\cN \( \cR_1 \( \del_\beta \) \( \Psi_{a, b}^\s  \) \) \\
         & \quad - \cR_1 \( \nabla^{\mathrm{LC}}_{\del_\alpha} \del_\beta \) \( \Psi_{a, b}^\s  \) - \cR_1 \( \del_\beta \) \( \nabla_{\del_\alpha}^\cN \Psi_{a, b}^\s  \) \\
         &= \nabla_{\del_\alpha} \( \cR_1 \( \del_\beta \) \( \Psi_{a, b}^\s  \) \) + O \( 1 \).
   \end{align}
   Thus, it is enough to show that $\nabla_{\del_\alpha}^\cN \( \cR_1 \( \del_\beta \) \( \Psi_{a, b}^\s  \) \)$ is uniformly bounded.

   First, when $\beta = 1$, we have
   \begin{equation}
         \cR_1 \( \del_r \) \( \Psi_{a, b}^\s \) = r^2 \( \nabla_{\del_r}^\cN \Psi_{a, b}^\s - \del_r \Psi_{a, b}^\s \) = r^2 \( \Pi_{E_x^\cN} - \id_{E_x^\cN} \) \circ \( \del_r \Pi_{E_x^\cN} \) \widetilde{\Psi}_{a, b}^\s,
   \end{equation}
   and thus
   \begin{align}
      \nabla_{\del_\alpha}^\cN \( \cR_1 \( \del_r \) \Psi_{a, b}^\s \)   &= \nabla_{\del_\alpha}^\cN \( r^2 \( \Pi_{E_x^\cN} - \id_{E_x^\cN} \) \circ \( \del_r \Pi_{E_x^\cN} \) \widetilde{\Psi}_{a, b}^\s \) \\
         &= \Pi_{E_x^\cN} \del_\alpha \( r^2 \( \Pi_{E_x^\cN} - \id_{E_x^\cN} \) \circ \( \del_r \Pi_{E_x^\cN} \) \widetilde{\Psi}_{a, b}^\s \) \\
         &= r^2 \Pi_{E_x^\cN} \circ \( \del_\alpha \Pi_{E_x^\cN} \) \circ \( \del_r \Pi_{E_x^\cN} \) \widetilde{\Psi}_{a, b}^\s,
   \end{align}
   and now, using \cref{eq:Pi_d_alpha_Pi,eq:d_beta_Pi_Pi}, we get that $\nabla_{\del_\alpha}^\cN \( \cR_1 \( \del_r \) \Psi_{a, b}^\s \)$, and thus $\( \nabla_{\del_\alpha}^\cN \cR_1 \) \( \del_r \)$, are uniformly bounded.

   When $\beta  \in \{ 2, 3 \}$, using \cref{eq:R1_formula}, we have
   \begin{equation}
      \cR_1 \( \del_\beta \) \Psi_{a, b}^\s (x) = r^2 \Pi_{E_x^\cN} \circ \Sigma_\beta \circ G_x^\cN \circ \( \D_x^\cN \)^* \widetilde{\Psi}_{a, b}^\s (x),
   \end{equation}
   and, since $\del_\beta r = 0$, we get
   \begin{align}
      \nabla_{\del_\alpha}^\cN \( \cR_1 \( \del_\beta \) \( \Psi_{a, b}^\s (x) \) \) &= r^2 \Pi_{E_x^\cN} \circ \del_\alpha \( \Pi_{E_x^\cN} \circ \Sigma_\beta \circ G_x^\cN \circ \( \D_x^\cN \)^* \widetilde{\Psi}_{a, b}^\s (x) \) \\
         &= r^2 \Pi_{E_x^\cN} \circ \( \del_\alpha \Pi_{E_x^\cN} \) \circ \Sigma_\beta \circ G_x^\cN \circ \( \D_x^\cN \)^* \widetilde{\Psi}_{a, b}^\s (x) \\
         & \quad + r^2 \Pi_{E_x^\cN} \circ \Sigma_\beta \circ \( \del_\alpha G_x^\cN \) \circ \( \D_x^\cN \)^* \widetilde{\Psi}_{a, b}^\s (x) \\
         & \quad + r^2 \Pi_{E_x^\cN} \circ \Sigma_\beta \circ G_x^\cN \circ \( \del_\alpha \D_x^\cN \)^* \widetilde{\Psi}_{a, b}^\s (x) \\
         & \quad + r^2 \Pi_{E_x^\cN} \circ \Sigma_\beta \circ G_x^\cN \circ \( \D_x^\cN \)^* \( \del_\alpha \widetilde{\Psi}_{a, b}^\s (x) \),
   \end{align}
   and using the same bounds as before, we get that each summand is $O \( 1 \)$.

   The proof is thus complete.
\end{proof}

\bigskip

\appendix

\section{Monopoles with real orthogonal and compact symplectic structure groups}
\label{app:SO/Sp}

Following \cites{hurtubise_construction_1989,hurtubise_monopoles_1990}, one can introduce further restrictions on the Nahm data to produce monopoles with structure groups reduced to $\rG \subset \SU (N)$. In the maximal symmetry breaking case this was proven in \cite{hurtubise_construction_1989}*{Section~5} for real orthogonal and compact symplectic groups and further discussed in \cite{hurtubise_monopoles_1990} for arbitrary compact Lie groups. The methods can easily be adapted to the case of nonmaximal symmetry breaking.

Here we outline the cases when $\rG$ is either a real orthogonal and compact symplectic groups only. The reduction to real orthogonal groups (from $\SU(N)$ to the included $\SO(N)$) is equivalent to the existence of a $\( \nabla^\cN, \Phi^\cN \)$-compatible real structure, $\widehat{C}$, on $E_x^\cN$, that is, a $\nabla^\cN$-parallel, real linear bundle map that commutes with $\Phi^\cN$, anti-commutes with the multiplication by $i$, and is an involution, that is $\widehat{C}^2 = \id_{E_x^\cN}$. Similarly, the reduction to compact symplectic groups is equivalent to the existence of a $\( \nabla^\cN, \Phi^\cN \)$-compatible quaternionic structure, $\widehat{J}$, on $E_x^\cN$, that is, a $\nabla^\cN$-parallel, real linear bundle map that anti-commutes with $\Phi^\cN$, anti-commutes with the multiplication by $i = \sigma_1$, and is a complex structure, that is $\widehat{J}^2 = - \id_{E_x^\cN}$.

Let $\cN$ be a Nahm data in temporal gauge and that for all $a \in \{ 1, \ldots, n - 1 \}$ satisfies
\begin{align}
   \lambda_a   &= - \lambda_{n+1 - a}, \\
   r_a^\pm     &= r_{n+1 - a}^\mp, \\
   r_a^0       &= r_{n+1 - a}^0,
\end{align}
and for all $b \in \{ 1, \ldots, r_a^\pm \}$
\begin{equation}
   k_{a, b}^\pm = - k_{n+1 - a, b}^\mp,
\end{equation}
and thus
\begin{align}
   k_a^\pm  &= k_{n+1 - a}^\mp, \\
   k_a      &= - k_{n+1 - a}, \\
   m_a      &= m_{n - a}.
\end{align}
Let $\iota \colon \rl \rightarrow \rl; t \mapsto - t$. Then $\iota \( \lambda_a, \lambda_{a + 1} \) = \( \lambda_{n - a }, \lambda_{n+1 - a} \)$.

\begin{itemize}
   
   \item[] \textbf{Orthogonal groups:} Let us further assume that for all $t \in \( \lambda_a, \lambda_{a + 1} \)$ and $\alpha \in \{ 1, 2, 3 \}$:
   \begin{equation}
      T_\alpha^{(a)} (t) = \overline{T_\alpha^{(n-a)} (- t)}.
   \end{equation}
   Let $c_a \colon \cx^{m_a} \otimes \bH \rightarrow \cx^{m_{n - a}} \otimes \bH \cong \cx^{m_a} \otimes \bH$ given via
   \begin{equation}
      c_a \( v \otimes q \) = \overline{v} \otimes {q},
   \end{equation}
   and then, for all $\Psi = \( \uppsi_a \)_{a = 1}^{n - 1} \in L^2 \( \bT \)$, let us define
   \begin{equation}
      C \( \Psi \) = \( c_{n - a} \( \uppsi_{n - a} \) \circ \iota \)_{a = 1}^{n - 1}.
   \end{equation}
   Then $C$ is a real structure on $L^2 \( \bT \)$ and for all $x \in \rl^3$, $C$ preserves $E_x^\cN$. Let $\widehat{C}_x \eqdef C|_{E_x^\cN}$ and consider $\widehat{C}$ as a bundle map. Then we have
   \begin{equation}
      \nabla^\cN \widehat{C} = 0, \quad \& \quad \left[ \Phi^\cN, \widehat{C} \right] = 0.
   \end{equation}
   Thus $\( \nabla^\cN, \Phi^\cN \)$ is an $\SO (N)$-monopole.

   \item[] \textbf{Compact symplectic groups:} Let us further assume that for all $t \in \( \lambda_a, \lambda_{a + 1} \)$ and $\alpha \in \{ 1, 2, 3 \}$:
   \begin{equation}
      T_\alpha^{(a)} (t) = - \overline{T_\alpha^{(n-a)} (- t)}.
   \end{equation}
   Let $j_a \colon \cx^{m_a} \otimes \bH \rightarrow \cx^{m_{n - a}} \otimes \bH \cong \cx^{m_a} \otimes \bH$ given via
   \begin{equation}
      j_a \( v \otimes q \) = \overline{v} \otimes \( q \sigma_2 \),
   \end{equation}
   and then, for all $\Psi = \( \uppsi_a \)_{a = 1}^{n - 1} \in L^2 \( \bT \)$, let us define
   \begin{equation}
      J \( \Psi \) = \( j_{n - a} \( \uppsi_{n - a} \) \circ \iota \)_{a = 1}^{n - 1}
   \end{equation}
   Then $J$ is a quaternionic structure on $L^2 \( \bT \)$ and for all $x \in \rl^3$, $J$ preserves $E_x^\cN$. Let $\widehat{J}_x \eqdef J|_{E_x^\cN}$ and consider $\widehat{J}$ as a bundle map. Then we have
   \begin{equation}
      \nabla^\cN \widehat{J} = 0, \quad \& \quad \left[ \Phi^\cN, \widehat{J} \right] = 0.
   \end{equation}
   Thus $N$ is necessarily even and $\( \nabla^\cN, \Phi^\cN \)$ is an $\Sp \( \tfrac{N}{2} \)$-monopole.

\end{itemize}

\bigskip

\section{Low rank examples}
\label{app:low_rank}

As an illustration of our results, we describe the Nahm data that yield $\SU (3)$-monopoles with nonmaximal symmetry breaking.

Since $N = 3$, $n$ can only take the values $2$ and $3$. When $n = 3$, the monopole has maximal symmetry breaking, and our construction above is equivalent with that of Hurtubise and Murray in \cite{hurtubise_construction_1989}. Thus, we focus of on the novel, $n = 2$ case. Now we have two imaginary eigenvalue, $i \lambda_1$ and $i \lambda_2$ and two ranks, $r_1$ and $r_2$, which have to satisfy that $r_1, r_2 \in \Z_+$ and $r_1 + r_2 = 3$. Without any loss of generality, we assume that $r_1 = 1$ and $r_2 = 2$. Then $\lambda_2 = - \tfrac{1}{2} \lambda_1 \in \rl_+$. Note furthermore, that we cannot have any zero Chern numbers (since both eigenvalues are at the ``ends''), thus $r_1 = r_1^-$ and $r_2 = r_2^+$. Because of that, we must have three Chern numbers, $k_{1, 1}^-, k_{2, 1}^+, k_{2, 2}^+ \in \Z_+$, so that $m_1 = k_{1, 1}^- = k_{2, 1}^+ + k_{2, 2}^+$. In summary, the Nahm data consists of a solution of Nahm's \cref{eq:nahm}, $\( T_1, T_2, T_3 \)$, on the single interval, $\( \lambda_1, \lambda_2 \)$, with the extra requirement that for all $t \in \( \lambda_1, \lambda_2 \)$ and $\alpha \in \{ 1, 2, 3 \}$, we have
\begin{equation}
   T_\alpha \( t \) = \tfrac{\rho_\alpha^-}{2 \( \lambda_1 - t \)} + \tfrac{\rho_\alpha^+}{2 \( \lambda_2 - t \)} + \cR_\alpha (t),
\end{equation}
where
\begin{itemize}
   \item The map
      \begin{equation}
         \widehat{\rho}^- \colon \su (2) \rightarrow \u \( m_1 \); \ x_1 \sigma_1 + x_2 \sigma_2 + x_3 \sigma_3 \mapsto x_1 \rho_1^- + x_2 \rho_2^- + x_3 \rho_3^-,
      \end{equation}
      is an irreducible representation.
   \item The map
      \begin{equation}
         \widehat{\rho}^\pm \colon \su (2) \rightarrow \u \( m_1 \); \ x_1 \sigma_1 + x_2 \sigma_2 + x_3 \sigma_3 \mapsto x_1 \rho_1^+ + x_2 \rho_2^+ + x_3 \rho_3^+,
      \end{equation}
      is a reducible representation that decomposes to a $k_{2, 1}^+ \oplus k_{2, 2}^+$-dimensional irreducible components.
   \item The function $\cR_\alpha$ extends continuously to $\left[ \lambda_1, \lambda_2 \right]$.
\end{itemize}

\smallskip

\begin{remark}
   When $k_{1, 1}^- = m_1 = 2$ and $k_{2, 1}^+ = k_{2, 2}^+ = 1$, then in the above construction, for all $\alpha \in \{ 1, 2, 3 \}$ and in some gauge, we have that $\rho_\alpha^- = \sigma_\alpha$ and $\rho_\alpha^+ = \begin{pmatrix} 0 & 0 \\ 0 & 0 \end{pmatrix}$.

   In particular, $\( T_1 (t), T_2 (t), T_3 (t) \)$ converges as $t \rightarrow \lambda_2$ from the left. This moduli space was first studied extensively by Dancer in \cites{Dancer-NahmHyperKahler,Dancer-SU3monopoles} and further with Leese in \cites{Dancer-Leese-dynamicsSU3monopoles,Dancer-Leese-numerical}.
\end{remark}



\bigskip

\begin{acknowledgment}
   The first author acknowledges the support of the Natural Sciences and Engineering Research Council of Canada (NSERC), RGPIN-2019-04375.

   The authors are thankful to Mark Stern for the many illuminating discussions about gauge theory, in general, and this project, in particular, that have greatly contributed to the completion of this paper. They also wish to thank Sergey Cherkis, Jacques Hurtubise, Gon\c{c}alo Oliveira, and Thomas Walpuski for their feedback and insight.

   While working on various parts of this project, the second author also enjoyed the hospitality of the University of Waterloo, the Perimeter Institute, the Fields Institute, and Duke University.

   Both authors enjoyed opportunities to advance the research of this paper at the conference ``Geometric and analytic aspects of moduli spaces'' in Hannover in 2019, and at the BIRS online workshop ``Geometry, Analysis, and Quantum Physics of Monopoles'' in 2021.
\end{acknowledgment}

   \bibliography{Bib-Benoit,Bib-monopoles,references}

\end{document}